\documentclass[11pt]{amsart}

\usepackage{pinlabel}
\usepackage{amsmath} 
\usepackage{cjhebrew}
\usepackage{graphicx} 
\setkeys{Gin}{keepaspectratio}
\usepackage[pdftex, plainpages=false, pdfpagelabels]{hyperref}
\usepackage{url}
\usepackage{bm}
\usepackage{mathabx}
\usepackage[left=1.25in,top=1in,right=1.25in,bottom=1in,head=.1in]{geometry}
\usepackage{xcolor}
\usepackage{tikz}
\usepackage{tikz-cd}
\usepackage{adjustbox}
\usepackage[normalem]{ulem} % 
\usepackage{enumitem}

\usepackage{verbatim}
\newcommand{\items}{\begin{itemize}[leftmargin=25pt,rightmargin=15pt]
  \setlength\itemsep{2pt}}
\newcommand{\stopitems}{\end{itemize}}
\usepackage[percent]{overpic}
\usepackage{mdframed, stmaryrd, subcaption}
\usepackage{mathrsfs}
\usepackage{pgfplots}
\usepackage{multicol}
\usepackage{float}
\usepackage{wrapfig}
\usepackage{todonotes}

\usepackage{fancyhdr}
\usetikzlibrary{arrows, cd, arrows.meta, patterns, intersections, pgfplots.fillbetween}

\NewDocumentCommand{\DIV}{om}{%
  \IfValueT{#1}{\setcounter{#2}{\numexpr#1-1\relax}}%
  \csname #2\endcsname
}

\newcommand{\R}{\mathbb{R}}

\newcommand{\W}{\mathbb{W}}

\newcommand{\set}[1]{\left\{#1\right\}}

\newcommand{\abs}[1]{\left|#1\right|}

\renewcommand{\a}{\alpha}

\def \co{\colon\thinspace}

\newtheorem{theorem}{Theorem}[section] % numbered theorems, lemmas, etc
\newtheorem*{theorem*}{Theorem}
\newtheorem{lemma}[theorem]{Lemma}

\newtheorem*{conjecture*}{Conjecture}
\newtheorem*{question*}{Question}
\newtheorem*{lemma*}{Lemma}
\newtheorem{proposition}[theorem]{Proposition}
\newtheorem{corollary}[theorem]{Corollary}
\newtheorem*{corollary*}{Corollary}
\theoremstyle{definition}
\newtheorem{definition}[theorem]{Definition}

\newtheorem{remark}[theorem]{Remark}

\newtheorem{example}[theorem]{Example}

\newtheorem*{example*}{Example}
\newtheorem*{remark*}{Remark}
\newtheorem*{remarks*}{Remarks}
\newtheorem*{addenda*}{Addenda}
\newtheorem*{construction*}{Construction}

\makeatletter
\newcommand{\tpitchfork}{%
  \vbox{
    \baselineskip\z@skip
    \lineskip-.52ex
    \lineskiplimit\maxdimen
    \m@th
    \ialign{##\crcr\hidewidth\smash{$-$}\hidewidth\crcr$\pitchfork$\crcr}
  }%
}
\makeatother

\theoremstyle{plain}
\newtheorem*{Thm:Self_Int_of_Boundary}{Theorem \ref{Thm:Self_Int_of_Boundary}}

\theoremstyle{plain}
\newtheorem*{lemma:limitations on simple fold maps}{Lemma \ref{lemma:limitations on simple fold maps}}

\theoremstyle{plain}
\newtheorem*{Thm:Main_Theorem_2025}{Theorem \ref{Thm:Main_Theorem_2025}}

\theoremstyle{plain}
\newtheorem*{thm:main theorem sharpness}{Theorem \ref{thm:main theorem sharpness}}

\theoremstyle{plain}
\newtheorem*{coro:gromov extension}{Corollary \ref{coro:gromov extension}}

\makeatletter
\@namedef{subjclassname@2020}{\textup{2020} Mathematics Subject Classification}
\makeatother


\usepackage[super,sort&compress]{natbib}
\setcitestyle{numbers,square}
\hypersetup{
    linktocpage=true,
    colorlinks=true,
    citecolor=blue,
    urlcolor=blue,
    linkcolor=blue,
    citebordercolor={1 0 0},
    urlbordercolor={1 0 0},
    linkbordercolor={.7 .8 .8},
    breaklinks=true,
    }

\title{A sharp lower bound on the self-intersections\\ of fold singularities}

\author{Joshua Drouin}
\address{Florida Polytechnic University}
\email{jdrouin@floridapoly.edu}
\date{\today}
\author{Liam Kahmeyer}
\address{Missouri Valley College}
\email{kahmeyerl@moval.edu}

\pgfplotsset{compat=1.17}
\begin{document}

\begin{abstract}
    For an oriented surface $S$, the singular set of a fold map $f:S\to\R^2$ is a collection of smooth curves, also known as fold singularities. We construct a sharp lower bound on the number of self-intersections of such fold singularities. This is done by first establishing a sharp lower bound on the number of self-intersections of the boundary of a surface immersed in $\R^2$. We then construct a sharp lower bound for the number of self-intersections of the singular set of a simple stable fold map of a surface to $\R^2$ by viewing the connected components of the singular set as the boundary components of smaller surface components, and invoking the previously constructed lower bound for the number of self-intersections of an immersed boundary. 
\end{abstract}

\maketitle  %\vspace{-0.25in}

\section{Introduction}

    The relationship between fold singularities and manifolds with boundary has been explored and developed by numerous authors. In \cite{El70}, Eliashberg described necessary conditions for a codimension 0 map to have prescribed fold singularities. In the even dimensional case, this prescription splits the manifold into two submanifolds with boundary such that the submanifolds have equal Euler characteristic. Therefore, in our study of singularities, it is necessary to examine manifolds with boundary. In particular, Guth \cite{Gu09} investigated immersions of surfaces with a singular boundary component, and he constructed a sharp lower bound on the number of self-intersections of the immersed boundary. 

    Additional contributions include the works of Gromov, Pignoni, M. Yamamoto, T. Yamamoto, and Ryabichev. In \cite{Gromov_2009}, Gromov finds lower bounds on the topological complexity of the singular set of generic smooth maps. In \cite{Pignoni_1991}, Pignoni connected the genus of a surface with its orientability and the properties of the singular set of a map of the surface into the plane. In \cite{Yamamoto_2003}, M. Yamamoto studied the winding numbers of boundaries of immersed surfaces, and in \cite{Yamamoto_2009}, he uses the work of Eliashberg to calculate the number of singular set components of fold maps between closed oriented surfaces. In \cite{T_Yamamoto_2016}, T. Yamamoto determines certain conditions for which a degree-$d$ stable map to the plane or $2$-sphere admits a prescribed singular set, and in \cite{T_Yamamoto_2022} he extends the work of Pignoni by examining weakly-admissibly homotopic maps of surfaces to the plane. In \cite{Ryabichev_2020} and \cite{Ryabichev_2023}, Ryabichev extends Eliashberg's $h$-principle to first smooth maps of surfaces with cusps and folds, and second to arbitrary generic smooth maps of smooth manifolds. 

    Unless otherwise stated, all manifolds and maps between them are assumed to be smooth. While the main results of this paper pertain to fold singularities, the proofs we present rely on the immersion of surfaces with boundary immersed into the plane. First, we construct a lower bound on the number of self-intersections of immersed boundaries, and then we utilize this result to construct a lower bound on the number of self-intersections of fold singularities. We begin by generalizing a result of Guth \cite[Proposition 1]{Gu09} to surfaces with multiple boundary components. Recall, a surface is called \emph{non-planar} if every connected component has genus $g\geq 1$.

    \begin{Thm:Self_Int_of_Boundary} %Our extension of Guth
        Let $S$ be an oriented, compact, non-planar, not necessarily path-connected, surface with boundary $\partial S$. Suppose that $\partial S$ has $k$ path components. If $f : S \rightarrow \R^2$ is an immersion, then $f(\partial S)$ has at least $4 - k- \chi(S)$ self-intersections. This estimate is sharp. 
    \end{Thm:Self_Int_of_Boundary}

    In order to state a lower bound on the number of self-intersections of fold singularities, we recall the following definitions. 
    
    \begin{definition}[Singular Points]
        Let $M$ be an $m$-dimensional manifold and $N$ be an $n$-dimensional manifold, with $m \geq n$. Also, let $f:M \rightarrow N$ be a smooth map. We say that $x\in M$ is \emph{regular} if at point $x$, we have $\mathrm{rank}(\mathrm{ker}(df)) = m-n$. A point $x \in M$ is \emph{singular} if it is not regular. We denote the set of all singular points of a map by $\Sigma$, and the image of $\Sigma$ by $f(\Sigma)$. We note that $\Sigma \subset M$ and $f(\Sigma) \subset N$.
    \end{definition}

    For this paper, we consider 2-dimensional manifolds, typically denoted by $S$, that are mapped to $\R^2$ via $f:S\to \R^2$ where $f$ only admits fold points. Recall that a \emph{fold point} for $f$ is a point $p$ in $S$, with coordinate representation $p=(x_1,x_2)$, such that ${f(p)=f(x_1,x_2)=(x_1,\pm x_2^2)}$. We refer to all such maps as \emph{fold maps}, and, in particular, we are interested in simple fold maps. 
    
    \begin{definition}[Simple Fold Maps]
        Similar to Saeki \cite{Saeki_1996} and Sakuma \cite{SAKUMA_1994}, we call a smooth map $f:S\to\R^2$ a \emph{simple fold map} if it is a fold map such that for every singular point $x \in S$, the connected component of $f^{-1}(f(x))$ containing $x$ intersects the singular set of $f$ only at $x$.
    \end{definition}    
    \begin{definition}[Simple Stable Fold Maps]
        A simple fold map is said to be \emph{stable} if and only if the image of the singular set is an immersion with only normal (transverse double point) crossings. 
    \end{definition}

    Unless otherwise stated, we assume that all simple fold maps are stable. To this end, we omit the word ``stable" when referring to simple stable fold maps, unambiguously calling them simple fold maps. Moreover, outlined in full detail in Section \ref{section:SplittingAndGraphConstruction}, simple fold maps induce a splitting on a surface, wherein we view the surface as a collection of disjoint path components and where the fold singularities coincide with the boundaries of these path components.
    
    \begin{definition}[Induced Splitting]
         Let $f\co S\to \R^2$ be a simple fold map of an oriented closed surface $S$. Let $S_+$ (respectively, $S_-$) denote the closure in $S$ of the set of points $x\in S$ such that $d_xf$ is a linear map preserving (respectively, reversing) orientation. We denote the \emph{splitting} of $S$ induced via $f$ by $S=S_+\cup S_-$. Furthermore, we denote the number of components of $S_{+}$ (respectively, $S_-$) by $\#|S_+|$ (respectively, $\#|S_-|$). 
    \end{definition}

    The fold singularities are thus viewed as $\Sigma=S_+\cap S_-$, and the number of components of $\Sigma$ is denoted by $|\Sigma|$. Surfaces with a prescribed splitting can also be viewed as graphs. More precisely, we associate a bipartite graph $G$ with $f$ where the vertices of $G$ correspond to the interiors of the path components $S_{+}$ and $S_{-}$. An edge in $G$ between two vertices exists if and only if the path components share a common fold boundary. We are now in a position to state the main results.

    \begin{Thm:Main_Theorem_2025}
        Let $f\co S\to \R^2$ be a simple fold map with splitting $S=S_+\cup S_-$ that has no planar components whose fold singular set is $\Sigma=S_+\cap\, S_-$. If $\abs{\#|S_+|-\#|S_-|}=n$, then the number of self-intersections of fold singularities has a lower bound given by
            \begin{align*}
                \Delta_{\Sigma}&\geq 4\max\set{\#|S_+|,\#|S_-|}-(\chi(S)/2+|\Sigma|)\\[0.1in]
                &=2(\rho(G)+1+n)-(\chi(S)/2+|\Sigma|)
            \end{align*}
        where $|\Sigma|$ is the number of path components of fold-singularities, and $\rho(G)$ is the number of edges in a spanning tree of the graph $G$ corresponding to $f$.
    \end{Thm:Main_Theorem_2025}

    We also prove that this lower bound is sharp when $g$, $|\Sigma|$, $\#|S_+|$, and $\#|S_-|$ are considered an \emph{admissible combination}. Discussed in greater detail in Section \ref{remark:admissible combo}, an admissible combination of $g$, $|\Sigma|$, $\#|S_+|$, and $\#|S_-|$ satisfies the conclusion of the following lemma. 

    \begin{lemma:limitations on simple fold maps}
        Let $S$ be a closed, oriented surface of genus $g$. If $f\co S\to \R^2$ is a simple fold map with splitting $S=S_+\cup S_-$ that has no planar components whose fold singular set, denoted by $\Sigma=S_+\cap S_-$, is nonempty, then the following are true:
            \begin{enumerate}[label=(\roman*),itemsep=0.15cm]
                \item $g\geq 2$
                \item $g>|\Sigma|\geq1$
                \item $|\Sigma|\geq\#|S_+|+\#|S_-|-1 \geq1$
                \item $|\Sigma|-g\equiv 1\mod 2$
            \end{enumerate}
    \end{lemma:limitations on simple fold maps}

    \begin{thm:main theorem sharpness}
        Let $g$, $|\Sigma|$, $\#|S_+|$, and $\#|S_-|$ be an admissible combination. If $S$ is a closed, oriented surface with genus $g$, then there exists a simple fold map $f:S\to\R^2$ with splitting $S=S_+\cup S_-$, that has no planar components whose fold singular set is $\Sigma=S_+\cap S_-$, such that the number of self-intersections of fold singularities is
            \[\Delta_{\Sigma}= 2(\rho(G)+1+n)-(\chi(S)/2+|\Sigma|)\]
        where $|\Sigma|$ is the number of path components of fold-singularities, $\rho(G)$ is the number of edges in a spanning tree of the graph $G$ corresponding to $f$, and $\abs{\#|S_+|-\#|S_-|}=n$.
    \end{thm:main theorem sharpness}

    Finally, Theorem \ref{Thm:Main_Theorem_2025} and Theorem \ref{thm:main theorem sharpness} can be restated in terms of the homology of the surface $S$ yielding the following refinement of Gromov’s \cite[Section 2.1]{Gromov_2009} homological lower bound.

    \begin{coro:gromov extension}
        Let $S$ be a closed orientable surface and $f:S \to \R^2$ a simple fold map. Then, 
            \[
                \Delta_{\Sigma}  \geq 2(\rho(G)+ n) +\frac{1}{2}|H_*(S)| -|\Sigma|
            \]
        where $|\Sigma|$ is the number of path components of fold-singularities, $\rho(G)$ is the number of edges in a spanning tree of the graph $G$ corresponding to $f$, $\Delta_{\Sigma}$ is the number of self-intersections of fold singularities, and $|H_*(S)|$ denotes the sum of the Betti numbers of $S$.
    \end{coro:gromov extension}

    The paper is structured as follows. In Section \ref{CodimZeroImm}, we define the notion of inner and outer components of boundaries of surfaces immersed in the plane, as well as review the winding number of a curve in the plane. We dedicate Section \ref{TheWhitneyFormula} to recall a renowned formula of Whitney, state its relevance to our paper, and provide examples of computing the winding number and verifying Whitney's formula for specific curves of interest. Next, in Section \ref{PositiveIsotopy} we define positive isotopy and concordance of curves, again providing examples relevant to this paper. In Section \ref{extending_guth}, we state and prove a generalized version of Guth's theorem. The purpose of Section \ref{section:SplittingAndGraphConstruction} is to recall how weighted bipartite graphs are associated to simple stable fold maps. Finally, in Section \ref{section:MainTheorems}, we state and prove the main results of this paper. The paper concludes with two applications of our main results: a detailed example and a refinement of Gromov's homological lower bound.  

    \textbf{Acknowledgments:}
    The authors would like to thank Alex Kepes for continued support and insightful comments and examples, and Luciana Scuderi for translating \cite{Ha60} of Haefliger, which is written in French. We are also grateful to Dave Auckly and Alex Joyce for useful discussions and advice.

\section{Codimension Zero Immersions}\label{CodimZeroImm}
    We begin by examining codimension zero immersions of an oriented $n$-manifold with boundary into $\R^n$. The specific focus of this paper will be the case $n=2$. Unless otherwise stated, we let $S$ be an oriented, compact, connected manifold of dimension $n$ with boundary $\partial S$, and $f\co S\to \R^n$ a codimension zero, orientation preserving, immersion such that the restriction $f|_{\partial S}$ is a codimension one self-transverse immersion of $\partial S$ to $\R^n$. Note, $\partial S$ may or may not be connected.

\subsection{Outer and inner components of the immersed boundary}\label{InnerAndOuterDefinitions}

    The immersed submanifold $f(\partial S)$ inherits a canonical coorientation. Namely, it is cooriented by the vector field $df(\nu)$, where $\nu$ is an outward normal vector field in $TS$ over $\partial S$. The submanifold $f(\partial S)$ also inherits an orientation such that the normal coorientation of $f(\partial S)$ followed by the orientation of $f(\partial S)$ agrees with the standard orientation on $\R^n$.  

\begin{definition}[Inner and Outer points]\label{Def:Generalized_Inner_and_Outer}
    Let $\gamma_1, ..., \gamma_k$ denote connected components of $\partial S$. Let $U(\gamma_i)$ denote the unique unbounded connected component of the complement to $f(\gamma_i)$ in $\R^n$. Let $x\in \partial U(\gamma_i)$ be a point that is not a double point of $f|_{\partial S}$. We say that $x$ is \emph{outer} if the coorientation of $f(\gamma_i)$ at $x$ points in the interior direction of $U(\gamma_i)$. Otherwise, we say that the point $x$ is \emph{inner}. 
\end{definition}

\begin{definition}[Inner and Outer immersed boundaries]\label{Definition: inner and outer boundaries}\label{def:inner/outer surfaces}
    We say that the entire immersed boundary component $f(\gamma_i)$ is \emph{outer} if all non-double points of $\partial U(\gamma_i)$ are outer. Otherwise, we say that $f(\gamma_i)$ is \emph{inner}. Note, that an inner component may have both inner points and outer points. 
\end{definition}

    Let $p_i \in \gamma_i$ such that $f(p_i)$ is a non-double point. We choose and fix a reference point $f(p_i)$ for each path component $f(\gamma_i)$ in such a way that if $f(\gamma_i)$ is an outer component, then $f(p_i)$ is outer, and if $f(\gamma_i)$ is an inner component, then $f(p_i)$ is inner. There is an alternate definition regarding whether a point is inner or outer.

\begin{lemma}\label{lemma: base point is minimum}
    Suppose that $p_i$ is a point on $\gamma_i$ such that $f(p_i)$ is not a double point and the $n^{\text{th}}$ coordinate of $f(p_i)$ assumes the minimum value. Then $f(p_i)$ is \emph{outer} if the coorientation of $f(\gamma_i)$ at $f(p_i)$ agrees with $-\frac{\partial}{\partial x_n}$. Otherwise, $f(p_i)$ is  \emph{inner}. 
\end{lemma}
\begin{proof}
     Let $p_i$ be a point on $\gamma_i$ such that the $n^{\text{th}}$ coordinate of $f(p_i) = (x_1,\cdots,x_n)$ attains a minimum value. Consider the $(n-1)$-dimensional hyperplane $\mathcal{P}$ containing the point $f(p_i)$ that is parallel to the coordinate hyperplane $\R^{n-1}$. The hyperplane $\mathcal{P}$ splits $\R^n$ into two regions: $R_1 = \{(y_1, \cdots y_n) \in \R^n ~ | ~ y_n>x_n \}$ and $R_2 = \{(y_1, \cdots y_n) \in \R^n ~ | ~ y_n<x_n\}$. Since $x_n$ is a minimum value for the $n^{\text{th}}$ coordinate of $f(p_i)$, we necessarily have that $f(\gamma_i) \subset R_1 \cup \mathcal{P}$ and $f(\gamma_i) \cap R_2 = \emptyset$. We also note that $R_2 \subset U(\gamma_i)$. 
     
     Now, suppose the coorientation of $f(\gamma_i)$ at $f(p_i)$ agrees with $-\frac{\partial}{\partial x_n}$. In this case, the coorientation of $f(\gamma_i)$ points into the interior of $R_2$ and thus points into the interior of $U(\gamma_i)$, making $f(p_i)$ an outer point. Next, suppose the coorientation of $f(\gamma_i)$ at $f(p_i)$ agrees with $\frac{\partial}{\partial x_n}$. In this case, the coorentation of $f(\gamma_i)$ at $f(p_i)$ points away from the interior of $R_2$ and therefore away from the interior of $U(\gamma_i)$, making $f(p_i)$ an inner point. 
\end{proof}

\begin{example}
    An example when $n=2$ can be seen in Figure \ref{fig:Outer_and_Inner_Definitions}. In Figure \ref{fig:Outer_and_Inner_Definitions_a}, a boundary component $\gamma_i$ of a compact surface is immersed into the plane. It inherits an orientation and coorientation as seen by the arrows; additionally it determines the unbounded region $U(\gamma_i)$. In Figure \ref{fig:Outer_and_Inner_Definitions_b}, part of the curve is orange and two points are highlighted. The orange refers to the portion of $f(\gamma_i)$ that is defined to be $\partial U(\gamma_i)$. It is only on this portion of the curve that we examine points to determine if the curve is inner or outer. 

    The two given points illustrate our definitions of inner and outer. The coorientation of $f(\gamma_i)$ at the point $f(p_1)$ is pointing \textit{outward} into $U(\gamma_i)$, meaning that $f(p_1)$ is an outer point. Furthermore, the coorientation of $f(\gamma_i)$ at the point $f(p_2)$ is \textit{not} pointing into $U(\gamma_i)$, meaning that $f(p_2)$ is an inner point. Since there exists an inner point on $\partial U(\gamma_i)$, the entire curve $f(\gamma_i)$ is an inner curve. 
\end{example}

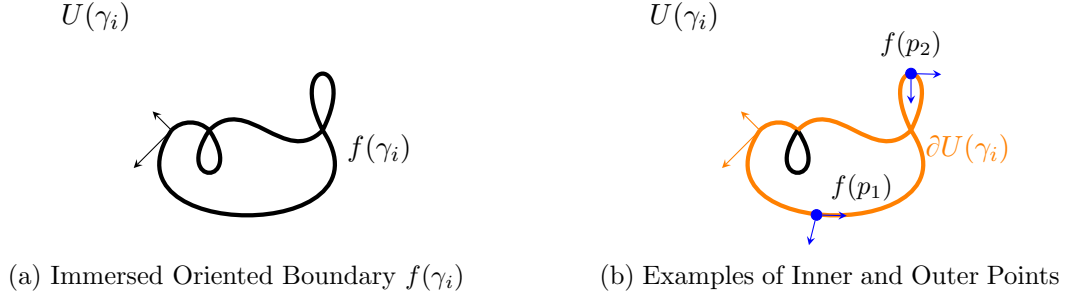
\begin{figure}[H]
    \begin{subfigure}{.4\textwidth}
        \centering 
        \begin{tikzpicture}[scale=1]
            %%% Lower Loop
            \draw[ultra thick] (-.5,0) .. controls (-1.5,-1.5) and (2.5,-1.5) .. (1.5,0);
            \draw[ultra thick] (-.5,0) to [bend left=45] (0,0);
            %%% Top Right Loop
            \draw[ultra thick] (1.5,0) .. controls (1,1) and (2,1) .. (1.5,0);
            %%% Top Left Loop
            \draw[ultra thick] (0,0) .. controls (.5,-.75) and (-.5,-.75) .. (0,0);
            %%% Connector Loop
            \draw[ultra thick] (0,0) .. controls (.5,.5) and (1,-.5) .. (1.5,0);

            %%% Coorientation Arrows
            \draw[->, >=stealth] (-.5,0) to (-.75,.25);
            \draw[->, >=stealth] (-.5,0) to +(-.5,-.5);

            %%% Labels
            \node[] (a) at (2.25, -.25) {$f(\gamma_i)$};
            \node[] (b) at (-1.5,1.5) {$U(\gamma_i)$};
            
        \end{tikzpicture}
        \caption{Immersed Oriented Boundary $f(\gamma_i)$}
        \label{fig:Outer_and_Inner_Definitions_a}
    \end{subfigure}
    \hspace{1.5cm}
    \begin{subfigure}{.4\textwidth}
        \centering 
        \begin{tikzpicture}[scale=1]
            %%% Lower Loop
            \draw[ultra thick, orange] (-.5,0) .. controls (-1.5,-1.5) and (2.5,-1.5) .. (1.5,0);
            \draw[ultra thick, orange] (-.5,0) to [bend left=45] (0,0);
            %%% Top Right Loop
            \draw[ultra thick, orange] (1.5,0) .. controls (1,1) and (2,1) .. (1.5,0);
            %%% Top Left Loop
            \draw[ultra thick] (0,0) .. controls (.5,-.75) and (-.5,-.75) .. (0,0);
            %%% Connector Loop
            \draw[ultra thick, orange] (0,0) .. controls (.5,.5) and (1,-.5) .. (1.5,0);

            %%% Coorientation Arrows
            \draw[->, >=stealth, orange] (-.5,0) to (-.75,.25);
            \draw[->, >=stealth, orange] (-.5,0) to +(-.5,-.5);

            %%% Labels
            \node[orange] (a) at (2.25, -.25) {$\partial U(\gamma_i)$};
            \node[] (b) at (-1.5,1.5) {$U(\gamma_i)$};

            %%% Points
            %Outer
            \node[shape=circle, draw=blue, fill=blue, minimum size=4pt, inner sep=0] (o) at (.25,-1.12) [label=45:\small$f(p_1)$] {};
            %Inner
            \node[shape=circle, draw=blue, fill=blue, minimum size=4pt, inner sep=0] (i) at (1.5,.75) [label=above:\small$f(p_2)$] {};

            %%% Outer Arrows
            \draw[->, >=stealth, blue] (o) to +(-.1,-.4);
            \draw[->, >=stealth, blue] (o) to +(.4,-.01);

            %%% Inner Arrows
            \draw[->, >=stealth, blue] (i) to +(0,-.4);
            \draw[->, >=stealth, blue] (i) to +(.4,-.01);

        \end{tikzpicture}
        \caption{Examples of Inner and Outer Points}
        \label{fig:Outer_and_Inner_Definitions_b}
    \end{subfigure}
        
        \caption{Immersed Boundaries with Coorientation and Orientation}
        \label{fig:Outer_and_Inner_Definitions}
    \end{figure} 

\subsection{The Gauss Map and the Winding Number}

Given an orientation preserving immersion $f\co S\to \R^n$ of a compact oriented manifold $S$ of dimension $n$, we have seen that the boundary $\partial S$ inherits a canonical orientation, while the immersed boundary $f(\partial S)$ inherits a canonical coorientation. Let $\nu\co \partial S\to S^{n-1}$ be a coorienting unit vector field over $f(\partial S)$. We say that $\nu$ is a \emph{Gauss map}. Since both $\partial S$ and $S^{n-1}$ are oriented and $f$ is orientation preserving, the degree of the Gauss map $\nu$ is well-defined. We emphasize that we do not assume that $\partial S$ is path-connected, and define the degree of $\nu$ as the algebraic number of preimages of a regular point of $\nu$. 
We call the degree of $\nu$ the \emph{winding number} of the immersed boundary $\partial S$ and denote it by $w(\partial S)$.   

The winding number allows us to compare invariants of a manifold to the invariants of its boundary. The following theorem is a version of the Poincar\'e-Hopf Theorem, see \cite[Section 6]{Mi65}.

\begin{theorem}  The winding number $w(\partial S)$ equals the Euler characteristic $\chi(S)$ of the manifold $S$. 
\end{theorem}

\section{The Whitney Formula}\label{TheWhitneyFormula} 
    The Whitney formula~\cite{Wh37} allows us to relate the winding number of a closed, connected, oriented, immersed curve $\gamma$ in $\R^2$ with its self-intersections. We will parameterize $\gamma$ as $\gamma(t):S^1\to\R^2$ by defining $S^1$ as $[0, 1]/\sim$, where $0$ and $1$ are identified. The base point, $\gamma(0)$ is chosen following the conventions of Section \ref{InnerAndOuterDefinitions}. That is, if $\gamma$ is an outer (respectively, inner) curve, then $\gamma(0)$ is an outer (respectively, inner) point. 
    
    A \emph{self-intersection point}, or \emph{double point}, of $\gamma(t)$ is defined to be a pair $0<t_1< t_2<1$ such that $\gamma(t_1)=\gamma(t_2)$. We emphasize that a generic curve $\gamma$ is self-transverse, and has no triple self-intersection points, i.e., there are no distinct triples of points $t_1, t_2$ and $t_3$ such that $\gamma(t_1)=\gamma(t_2)=\gamma(t_3)$. We say that a self-intersection point $x=\gamma(t_1)=\gamma(t_2)$ is \emph{positive} if the ordered pair of vectors $(\dot\gamma(t_1), \dot\gamma(t_2))$ forms a positive basis for $T_x\R^2$. Otherwise the self-intersection point is \emph{negative}. Note, the sign of a self-intersection point of $\gamma$ essentially depends on the orientation of $\gamma$ as well as on the base point $\gamma(0)$. 

    Let $n^+=n^+(\gamma)$ denote the number of positive self-intersection points of $\gamma$, and $n^-=n^-(\gamma)$ the number of negative self-intersection points. Additionally, let $i^+=i^+(\gamma)$ equal 1 if $\gamma$ is an outer curve and 0 if $\gamma$ is an inner curve. Similarly, let $i^-=i^-(\gamma)$ equal 1 if $\gamma$ is an inner curve and 0 if $\gamma$ is an outer curve. Lastly, let $w(\gamma)$ denote the winding number of $\gamma$. With this notation, we may finally state the Whitney formula.
    
    \begin{theorem}[Whitney \cite{Wh37}]
        If $\gamma$ is a 1-dimensional, closed, connected, oriented, immersed curve with $n^+$ positive self-intersection points and $n^-$ negative self-intersection points, then the winding number of $\gamma$ is $w(\gamma)=i^+-i^-+n^--n^+$.
    \end{theorem}
    We comment that in Whitney's paper \cite{Wh37}, the notions of positive and negative for self-intersections are opposite that of ours in this context. Since the winding number is additive over connected components, Whitney’s formula immediately gives the following.
    \begin{corollary}
        Let $\gamma$ be a closed, oriented, immersed curve with ordered, pointed components. Then 
        \begin{equation}\label{generalized whitney formula}
            w(\gamma)=I^+-I^-+N^--N^+,
        \end{equation}
        where $I^+$ (respectively, $I^-$) is the number of outer (respectively, inner) components and $N^+$ (respectively, $N^-$) is the number of positive (respectively, inner) self-intersection points.        
    \end{corollary}

    For non-connected curves, we will distinguish between the intersection points of different connected components, and the self-intersection points connected components. Let $\gamma$ be a curve with $k$ connected components, labeled $\gamma_1$, $\ldots$, $\gamma_k$. We define an intersection point between $\gamma_i$ and $\gamma_j$, with $i < j$ to be \emph{positive} if the orientation of $\gamma_i$ followed by the orientation of $\gamma_j$ agrees with the standard basis of $\R^2$. Otherwise, we say an intersection point between $\gamma_i$ and $\gamma_j$ is \emph{negative}. We denote the number of positive (respectively, negative) intersection points between different connected components $\gamma_i$ and $\gamma_j$ with $i<j$ by $N^+_{ij}$ (respectively, $N^-_{ij}$). We note that the algebraic number of intersection points between different components is zero, i.e., $N^+_{ij}=N^-_{ij}$. Additionally, we denote the total number of positive (respectively, negative) self-intersection points across all $\gamma_i$ by $N^+=\sum_{i}n^+(\gamma_{i})$ (respectively, $N^-=\sum_{i}n^-(\gamma_{i})$).

    \begin{table}[H]\renewcommand{\arraystretch}{1.5}
        \begin{tabular}{c|c|c|c|c|c}
        Curve $\gamma$ & $w(\gamma)$ & $i^+(\gamma)$ & $i^-(\gamma)$ & $n^-(\gamma)$ & $n^+(\gamma)$ \\[2.5pt] \hline 
        $A_k^-$        & $k-1$       & $0$           & $1$           & $k$           & $0$           \\[2.5pt] \hline 
        $A_k^+$        & $1-k$       & $0$           & $1$           & $1$           & $k-1$         \\[2.5pt] \hline 
        $B_{k,1}^-$    & $k-2$       & $0$           & $1$           & $k$           & $1$           \\[2.5pt] \hline 
        $B_{k,1}^+$    & $2-k$       & $1$           & $0$           & $1$           & $k$           \\[2.5pt] \hline 
        $C_k^-$        & $k+1$       & $1$           & $0$           & $k$           & $0$           \\[2.5pt] \hline 
        $C_k^+$        & $-k-1$      & $0$           & $1$           & $0$           & $k$           \\[2.5pt]
        \end{tabular}
        \caption{Invariants of key families of curves}
        \label{table:invariants of curves}
    \end{table}

    \begin{example}\label{ex: ABC curves} We highlight six families of curves and their invariants. These invariants are shown in Table \ref{table:invariants of curves}. Four specific curves when $k=3$ are shown in Figure \ref{fig:1}. These families will be utilized to show sharp lower bounds on estimates later in the paper. The first two families are denoted $A_k^-$ and, by reversing the (co)orientation, $A_k^+$. We note that the choices of basepoints of $A_k^-$ and $A_k^+$ do not agree. In particular, it is not true that $n^\pm(A_k^-)=n^\mp(A_k^+)$. The next two families are denoted $C_k^-$ and, by reversing the (co)orientation, $C_k^+$. Now, we choose the basepoints of $C_k^-$ and $C_k^+$ so that the basepoints agree. We denote the last two families by $B_{k,1}^-$ and, by reversing the (co)orientation, $B_{k,1}^+$. Again, we choose the basepoints of $B_{k,1}^-$ and $B_{k,1}^+$ so that the basepoints agree.
    \end{example}

    \begin{figure}[H]
        \centering
        \includegraphics[height=1.3in]{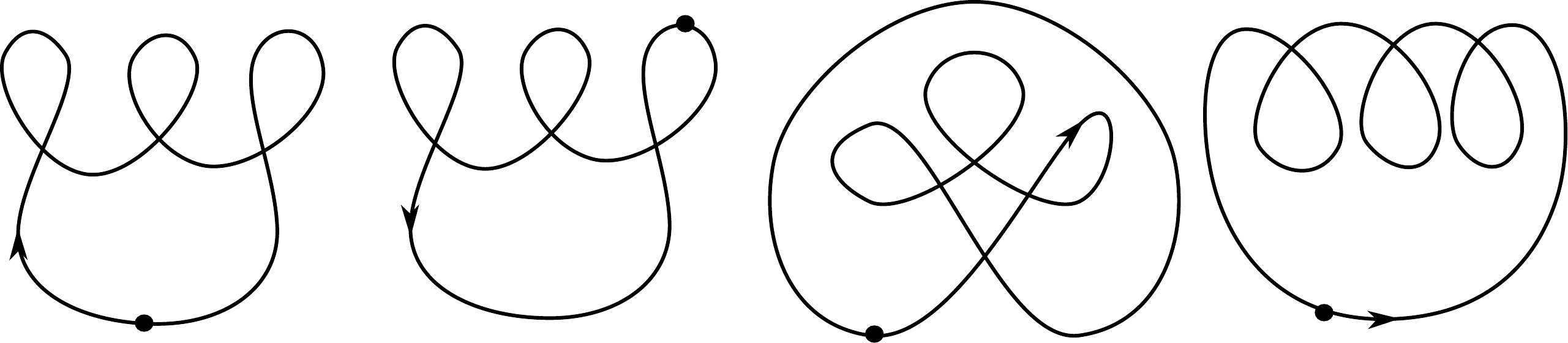}
        \caption{Curves $A_3^-$, $A_3^+$, $B_{3,1}^+$ and $C_3^-$.}
        \label{fig:1}
    \end{figure}

    \begin{remark}\label{B curves are important}
        We highlight that the family of curves $B^+_{k,1}$ are those Guth~\cite{Gu09} constructed to show the sharpness of his result. In particular, given a surface $S$ with a single boundary component, Guth showed (see Theorem \ref{th:Gu}) that there exists an immersion $f$ such that $f(\partial S)$ has exactly $2g+2$ self-intersections. The curve $B^{+}_{2-\chi(S),1}$ is precisely the image $f(\partial S)$.
    \end{remark}

\section{Positive Isotopy and Concordance of Curves}\label{PositiveIsotopy}

    Recall that every oriented immersed curve $\gamma\co S^1\to \R^2$ is canonically cooriented by a vector field such that, at every point on $\gamma$, the coorientation followed by the orientation forms a positive basis for $\R^2$.

    \begin{definition}[Positive Isotopy]
        Let $\gamma_i\co S^1\to \R^2$ be two cooriented curves, for $i=0,1$. We say that $\gamma_0$ and $\gamma_1$ are \emph{positively isotopic} if there is an immersion $f\co S^1\times [0,1]\to \R^2$ such that $f|_{S^1\times \{i\}}=\gamma_i$,  and the coorientations of $\gamma_0$ and $\gamma_1$ agree with $df(\frac{\partial}{\partial t})$, where $t$ is the standard coordinate on $[0,1]$. More generally, a closed curve $\gamma_0$ is positively isotopic to a closed curve $\gamma_1$ if the path components of $\gamma_0$ are in bijective correspondence with path components of $\gamma_1$ and the path components of $\gamma_0$ are positively isotopic to the corresponding path components of $\gamma_1$. If a curve $\gamma_i$ is positively isotopic to a curve $\gamma_j$, then we will write $\gamma_i \Rightarrow \gamma_j$.
    \end{definition}

    \begin{remark}
        There is a natural contact-geometric interpretation of positive isotopy. A cooriented immersed curve $\gamma\subset\R^2$ has a Legendrian conormal lift $\Lambda_\gamma\subset \text{ST}^{*}\R^2\cong J^1(S^1)$. Under the standard contact form, a family $\gamma_t$ is positive in our sense exactly when the lifts $\Lambda_t$ form a non-negative, respectively positive, Legendrian isotopy. Thus our positive isotopies are a restricted class of non-negative Legendrian isotopies whose wavefront projections remain immersed. Further details can be found in \cite{Ng_2005}.
    \end{remark}

    \begin{figure}[h]
        \centering
        \includegraphics[width=0.75\linewidth]{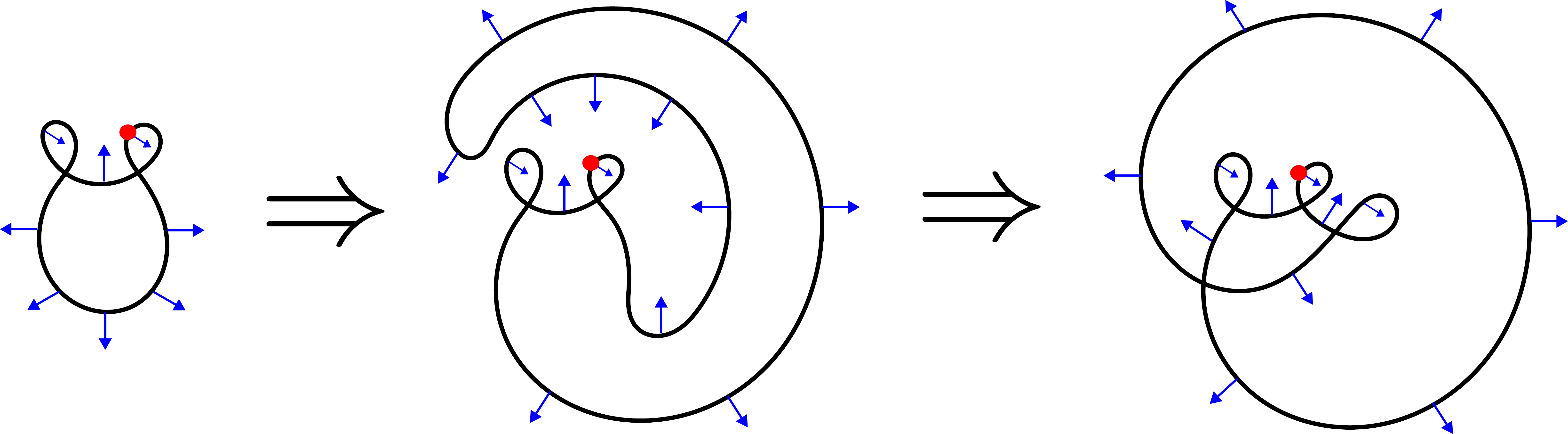}
        \caption{An inner curve that is positively isotopic to an outer curve.}
        \label{fig:c0 to b3,1}
    \end{figure}

    If $\gamma_0$ is positively isotopic to $\gamma_1$, an inner component, with respect to $\gamma_0$, can become an outer component, with respect to $\gamma_1$. An example of this is shown in Figure \ref{fig:c0 to b3,1} where the left-most curve is an inner curve, the left-most curve is positively isotopic to the right-most curve, and the right-most curve is an outer curve. Furthermore, we emphasize that the given reference point in the left-most curve is no longer a valid reference point in the right-most curve.

    \begin{lemma}\label{le:1}\label{lemma:neg_pos_iso}
     If a curve $\gamma_i$ is positively isotopic to a curve $\gamma_j$, then the curve $- \gamma_j$ is positively isotopic to the curve $- \gamma_i$, where $-\gamma$ denotes the curve obtained from $\gamma$ by reversing both orientation and coorientation.
    \end{lemma}

    The proof of Lemma \ref{lemma:neg_pos_iso} is straightforward, so we omit it. In Section \ref{section:MainTheorems} we shall use the following elementary isotopies between the standard families of curves introduced in Example \ref{ex: ABC curves}.

    \begin{example}\label{Positive_Isotopy_Contrapositive}  Two examples using Lemma \ref{lemma:neg_pos_iso}:
        \begin{itemize}[leftmargin=*]
           \item The curve $A_k^-$ is positively isotopic to $C^-_{k-2}$. The positive isotopy essentially consists of a Whitney type move eliminating one negative and one positive crossing. By Lemma~\ref{le:1}, the curve $C^+_{k-2}$ is positively isotopic to the curve $A_k^+$.
            \item The curve $A_k^+$ is positively isotopic to $B^+_{k+1,1}$.  To prove the latter, it is again helpful to apply Lemma~\ref{le:1}, and construct instead a positive isotopy from $B^-_{k+1,1}$ to $A_k^-$. For the picture when $k=2$, see Figure \ref{fig:c0 to b3,1}. 
        \end{itemize}
    \end{example}

    % \begin{figure}[h]
    %     \centering
    %     \includegraphics[width=0.8\linewidth]{Figures/Bk_conc.pdf}
    %     \caption{A positive concordance of $C_0^- \sqcup A_2^+$ that results in $B_{3,1}^+$}
    %     \label{fig:Bk conc}
    % \end{figure}

    \begin{definition}[Positive Concordance]
        For $i=0,1$, let $\gamma_i$ be closed curves, and let $f_i: \gamma_i \rightarrow \R^2$ be immersions. We say that $\gamma_0$ is \emph{positively concordant} to $\gamma_1$ if there is a compact, oriented surface $S$ with boundary such that $\partial S = \gamma_0 \cup \gamma_1$, and an immersion $F:S\rightarrow \R^2$ such that $F|_{\gamma_i}=f_i$, and the coorientation of $F(\gamma_{0})$ agrees with the inward normal coorientation of $\gamma_0$ in $S$ , and the coorientation of $F(\gamma_1)$ agrees with the outward normal coorientation of $\gamma_{1}$ in $S$.
    \end{definition}

    We note that if $\gamma_0$ is positively concordant to $\gamma_1$, then the number of path components of $\gamma_0$ may or may not be the same as the number of path components of $\gamma_1$. We also note that, under positive concordance, the base point of a pointed curve may or may not remain a valid base point.

    \begin{figure}[h]
        \centering
        \includegraphics[width=0.8\linewidth]{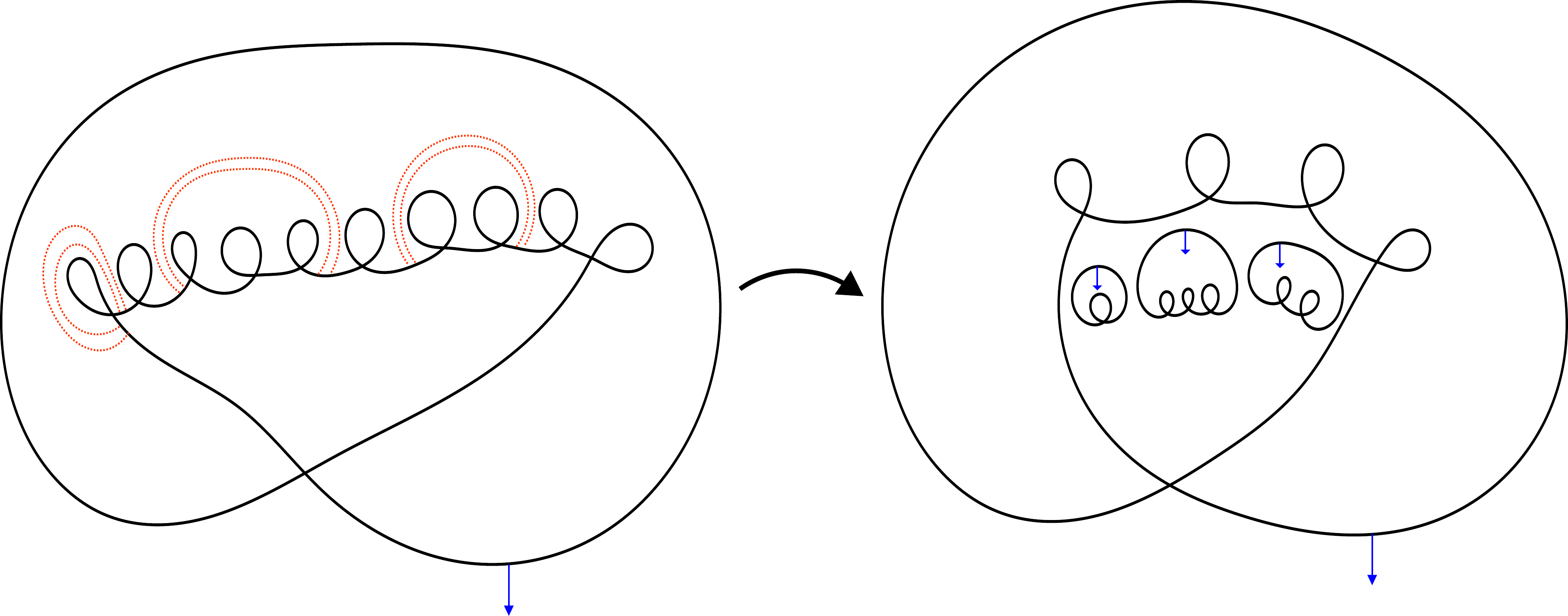}
        \caption{A positive concordance from $B^+_{10,1}$ to $B^+_{4,1} \sqcup C^+_1 \sqcup C^+_3 \sqcup C^+_2$}
        \label{fig:ConcDecomp}
    \end{figure}

    \begin{example}\label{splitting_Bs_example} For any decomposition $k=i_0+\cdots + i_s$ of a positive integer $k$ into the sum of positive numbers $i_j$, the curve $B_{k,1}^+$ is positively concordant to the disjoint union of curves $B_{i_0,1}^+ \sqcup C_{i_1}^+\sqcup \dots \sqcup C_{i_s}^+$. In Figure \ref{fig:ConcDecomp}, we see such a decomposition for $k=10$ with $i_0=4$, $i_1=1$, $i_2=3$, and $i_3=2$.  
    \end{example}

    To conclude this section, we state two useful lemmas related to positive isotopy and positive concordance whose proofs are immediate. 
    
    \begin{lemma}\label{guth_extension_lemma_1}
        Let $\gamma_0$ be a cooriented closed immersed curve in $\R^2$. Let $S$ be a compact surface of genus $g$ with non-empty boundary and let $f_{0}: S\to\R^2$ be an immersion such that $f_0(\partial S)=\gamma_0$. Suppose the coorientation of $\gamma_0$ agrees with $df_0(v)$, where $v$ is a vector field in $TS|_{\partial S}$ directed into the interior of $S$. If a closed cooriented curve $\gamma_1$ is positively isotopic to $\gamma_0$, then there is an immersion $f_1\co S\to \R^2$ such that $f_{1}|_{\partial F}$ is a parametrization of $\gamma_1$, and the coorientation of $\gamma_1$ agrees with $df_1(v)$.
    \end{lemma}

    \begin{lemma}\label{l:7}\label{guth_extension_lemma_2} 
        Let $\gamma_0$ be a cooriented closed immersed curve in $\R^2$. Let $S$ be a compact surface of genus $g$ with non-empty boundary and let $f_{0}: S\to\R^2$ be an immersion such that $f_0(\partial S)=\gamma_0$. Suppose the coorientation of $\gamma_0$ agrees with $df_0(v)$, where $v$ is a vector field in $TS|_{\partial S}$ directed into the interior of $S$. Let $x$ and $y$ be two points in $\partial S$, and $I$ be an immersed curve in $\R^2$ parameterized by $t\in [0,1]$ from $f_0(x)$ to $f_0(y)$ such that $\dot I(0)=df_0(v(x))$ and $\dot I(1)=-df_0(v(y))$. Let $F$ denote a surface obtained from $S$ by attaching a handle of index $1$ with attaching sphere $S^0=\{x, y\}$. Then $\gamma_0$ is positively concordant to a closed cooriented curve $\gamma_1$ such that there is an immersion $f_1\co F\to \R^2$ such that $f_{1}|_{\partial S}$ is a parametrization of $\gamma_1$, with $f_1|_{\partial F}=\gamma_1$, where the coorientation of $\gamma_1$ agrees with $df_1(u)$, where $u$ is a vector field in $TF|_{\partial F}$ directed into the interior of $F$.
    \end{lemma}

\section{An extension of the Guth theorem/Pinon theorem}\label{extending_guth}

    In an unpublished preprint, Guth determines the minimal number of self-intersection points of the boundary of an immersed surface. These immersions are considered \emph{normal}, if they are smooth, generic maps, such that the double points of the immersed boundary intersect transversely. 

    \begin{theorem}[Guth \cite{Gu09}]\label{th:Gu} Suppose $S$ is a compact, oriented surface of genus $g\geq 1$, with a single boundary component. If $f$ is a normal immersion $f\co S\to \R^2$, then $f(\partial S)$ has at least $2g+2$ self-intersections. Moreover, for each $g$, there exists an immersion $f_{g}$ such that $f_g(\partial S)$ has exactly $2g+2$ self-intersections.
    \end{theorem}

    \noindent We note that the Euler characteristic $\chi(S)$ of a surface of genus $g$ with one boundary component is $1-2g$. Thus, the estimate in the Guth theorem can also be written as $2g+2=3-\chi(S)$. 
    
    We now wish to extend this to a surface with multiple boundary components, and, to do this, we will utilize a result from Haefliger. Let $S$ be a compact surface with boundary. The surface $S$ as well as its boundary may or may not be path-connected. Let $f\co S\to \R^2$ be an immersion. Then, by \cite[Theorem 3]{Ha60}, the winding number $w$ of $f|\partial S$ satisfies the condition $\chi(S)= w(f|\partial S)$. Recall, the generalized Whitney formula in Equation \ref{generalized whitney formula}. Using Haefliger's result, we may rewrite the generalized Whitney formula as:
        \begin{equation}\label{Haefliger Result}
            \chi(S)=I^+-I^-+N^--N^+
        \end{equation}

\begin{proposition} Let $S$ be a compact, oriented, non-planar surface of genus $g$ with $k\geq 1$ boundary components and let $f\co S\to \R^2$ be a normal immersion. Then $N^+ \geq N^-$ and \ $N^+ \equiv N^-\mod 2$.
Moreover, the number
$\Delta=N^++N^-+N^+_{ij}+N^-_{ij}$ of self-intersection points of $f:S\to\R^2$ is even.
\end{proposition}
\begin{proof} Since $I^- = k - I^+$ and $\chi(S) = 2 - 2g - k$, 
the generalized Whitney formula, Equation \ref{Haefliger Result}, may be re-written as
        \[
            N^+-N^-=2g+2I^+-2.
        \]
Since $I^+\ge 0$ and $g\ge 1$, we deduce that $N^+\ge N^-$. Finally, we recall that $N^+_{ij}=N^-_{ij}$, and conclude that $\Delta \equiv 0 \mod 2.$        
\end{proof}

\subsection{Supporting Lemmas}
    In order to extend Guth's theorem to surfaces with multiple boundary components, we need to analyze numerous cases. These cases depend on both the number of immersed outer boundaries and whether they are embedded or not. We begin with the case where there is at least one immersed outer boundary with self-intersections. While not called an immersed outer boundary by Guth, if the surface has only one boundary component, then it must be an immersed outer boundary and have at least one negative self-intersection.

    \begin{lemma}\label{OnlyOneMeansOuter}
        If $S$ is an orientable compact surface with $1$ boundary component immersed into $\R^2$, then its immersed boundary component is outer.
    \end{lemma}
    \begin{proof}
        If $\gamma$ is inner, then there is an inner point $x$ on $\gamma$. This implies that the unbounded region $U$ of $\R^2\setminus \gamma$ has the property that the inverse image of any point of $U$ is non-empty, which contradicts the compactness of $S$.
    \end{proof}

    \begin{lemma}[Guth, \cite{Gu09}]\label{prop:18}\label{NegativeSelfIntAtLeast1}
        Let $f:S \rightarrow \R^2$ be an immersion of a compact, oriented surface $S$ with a single immersed boundary curve $\gamma$ such that $\gamma$ is an outer component with $n = n^+ + n^- > 0$ self-intersection points. Then, $n^- \geq 1$. In particular, the self-intersection point nearest to the specified base point is negative.
    \end{lemma}

    \begin{proof}
        Without loss of generality, we fix a base point $(x,y)=f(p) \in\gamma$ in such a way that the $y$ coordinate of $f(p)$ assumes the minimum value. Since $\gamma$ is outer, the coorientation of $\gamma$ at $f(p)$ agrees with $-\frac{\partial}{\partial y}$, by Lemma \ref{lemma: base point is minimum}, and, consequently, the orientation of $\gamma$ at $f(p)$ agrees with $\frac{\partial}{\partial x}$. 

        Now, starting at $f(p)$ and traveling along $\gamma$ in the direction of its orientation, denote the first self-intersection point by $z$. Suppose to the contrary that $z$ is a positive self-intersection point. Let $V$ be a small neighborhood in $\R^2$, centered on $z$ that contains no other self-intersections of $\gamma$. Consider now the intersection of $\gamma$ and $V$. Illustrated in Figure \ref{fig:positive selfint}, we see a portion of $\gamma$ along with the basepoint $f(p)$, self-intersection point $z$, coorientation arrows of $\gamma$, and neighborhood $V$ in dashed blue. Additionally, dotted arcs of $\gamma$ represent unknown behavior away from $V$ that has no impact on this result. We see that $\gamma$ partitions $V$ into four regions, $T$, $B$, $L$, and $R$. Recall, from Definition $\ref{def:inner/outer surfaces}$, that the coorientation points to the side of $\gamma$ opposite where the surface is immersed. Thus, the surface is immersed into every region, notably into region $B$. Since $f(p)$ assumes the minimum value of $y=0$, it means that $B$ is unbounded. However, the surface $S$ is compact. Thus, $z$ must be a negative self-intersection point. 
    \end{proof}

    While Lemma \ref{NegativeSelfIntAtLeast1} specifically applies to cases where the immersed boundary of $S$ is a single outer component, a similar proof shows that immersed outer boundaries, with self-intersections, must contain a negative self-intersection. This is true even for an arbitrary amount of immersed boundaries components.

    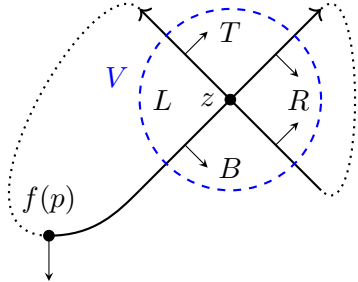
\begin{figure}[H]
            \centering
            \begin{tikzpicture}[scale=1.2]
                \draw[thick] (0,0.5) .. controls (0.5,0.5) and (0.75,0.75) .. (1,1);
                \draw[->, thick] (1,1) -- (3,3);
                \draw[->, thick] (3,1) -- (1,3);

                \draw[dashed, blue, thick] (2,2) circle (1cm);

                % Basepoint with Coorient
                \node[shape=circle, draw=black, fill=black, minimum size=4pt, inner sep=0] (o) at (0,0.5) {};
                \node[] (fp) at (0, 0.9) {$f(p)$};
                \draw[->, >=stealth] (0,0.5) to +(0,-0.5);

                % Self-int z and nhbd V
                \node[shape=circle, draw=black, fill=black, minimum size=4pt, inner sep=0] (o) at (2,2) {};
                \node[] (z) at (1.75, 2) {$z$};
                \node[blue] (V) at (0.75, 2.25) {$V$};

                % Other Coorient
                \draw[->, >=stealth] (1.5,1.5) to +(0.25,-0.25);
                \draw[->, >=stealth] (2.5,2.5) to +(0.25,-0.25);
                \draw[->, >=stealth] (1.5,2.5) to +(0.25,0.25);
                \draw[->, >=stealth] (2.5,1.5) to +(0.25,0.25);

                % Regions
                \node[] (T) at (2, 2.75) {$T$};
                \node[] (B) at (2, 1.25) {$B$};
                \node[] (L) at (1.25, 2) {$L$};
                \node[] (R) at (2.75, 2) {$R$};

                % Dotted Connections to finish curve
                \draw[dotted,thick] (3,3) .. controls (3.5,3.5) and (3.5,0.5) .. (3,1);
                \draw[dotted,thick] (1,3) .. controls +(-0.5,0.5) and +(-1.25,0) .. (0,0.5);
                
            \end{tikzpicture}
            \caption{Positive self-intersection}
            \label{fig:positive selfint}
        \end{figure}

    \begin{corollary}\label{corollary: outer with self int}
        Let $f:S \rightarrow \R^2$ be an immersion of a compact, oriented surface $S$ with at least one immersed outer boundary curve $\gamma$. If $\gamma$ contains self-intersection points, then at least one of them is negative. In particular, the self-intersection point nearest to the specified base point is negative.
    \end{corollary}
    \begin{proof}
        Suppose to the contrary that an immersed outer boundary $\gamma$ contained no negative self-intersections. Since $\gamma$ is outer, a chosen basepoint $f(p)\in\gamma$ is also outer, and a schematic of $\gamma$ is shown in Figure \ref{fig:positive selfint}. Thus, one (or both) of the points corresponding to the coorientation arrows pointing towards $R$ must be an inner point. However, this contradicts the assumption that $\gamma$ was outer. 
    \end{proof}

    Immersed outer boundaries intuitively represent the ``outer most" bounding curve of an immersed surface. It is natural to view any immersed surface as requiring an ``outer most" bounding curve, but it is in fact not necessary. When we increase the amount of boundary components, we are not always guaranteed to have an outer component. An immersion of a surface with more than one boundary component cannot have only a \textit{single} inner component, but can be comprised of \textit{only} inner components. An example of this can be seen in Figure \ref{fig:OnlyInnerComponents} where two inner curves, with their basepoints and coorientation, represent the immersed inner boundaries of a punctured disk. Let us now consider the case where all immersed boundaries are inner.

    \begin{figure}[H]
        \centering
            \begin{tikzpicture}[scale=1.5, rotate=90]
            \clip(-.75,-1.25) rectangle (.75,1.75);
            %\draw[help lines,xstep=1,ystep=1] (-5,-10) grid (18,6.5);
            %    \foreach \x in {-5,-4,...,18} { \node [anchor=north] at (\x,0) {\x}; }
            %    \foreach \y in {-10,-9,...,6} { \node [anchor=east] at (0,\y) {\y}; }
            %%% Curve 1
            \draw[ultra thick, blue] (0,0) .. controls (-2,2) and (2,2) .. (0,0);
            \draw[ultra thick, blue] (0,0) .. controls (-1,-1) and (1,-1) .. (0,0);
            
            %%% (Co)Orientation Arrows
            \node[shape=circle, draw=black, fill=black, minimum size=4pt, inner sep=0] (o) at (0,-.75) {};
            \draw[->, >=stealth] (0,-0.75) to (0,-.5);

            %%% Curve 2
            \draw[ultra thick, orange] (0,0.5) .. controls (-1,1.5) and (1,1.5) .. (0,0.5);
            \draw[ultra thick, orange] (0,0.5) .. controls (-2,-1.5) and (2,-1.5) .. (0,0.5);
            
            %%% (Co)Orientation Arrows
            \node[shape=circle, draw=black, fill=black, minimum size=4pt, inner sep=0] (o) at (0,1.25) {};
            \draw[->, >=stealth] (0,1.25) to (0,1);
            
        \end{tikzpicture}
        \caption{Immersion with only Inner Components}
        \label{fig:OnlyInnerComponents}
    \end{figure}
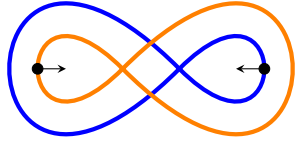

    \begin{definition}[Inner and Outer Arcs]\label{definition: inner and outer arcs}
         Given an immersed boundary component $\gamma$, a \emph{boundary arc} of $\gamma$ is defined as a maximal section of $\partial U(\gamma)$ between self-intersection points of $\gamma$ that contains no self-intersections. Recall the definition of inner and outer from Section \ref{InnerAndOuterDefinitions}. A boundary arc will be denoted as an \emph{inner arc}, respectively an \emph{outer arc}, if all points on the boundary arc are inner points, respectively outer points. Note, a boundary arc cannot have both inner and outer points. Furthermore, we emphasize that if a curve is an inner curve then it must have at least one inner arc. 
    \end{definition}

    \begin{definition}[Bounded Arcs]\label{definition: bounded arcs}
        Let $\gamma$ be an immersed boundary with $k$ connected components $\gamma_i$ for $1 \leq i \leq k$. An inner arc $\a$ on $\gamma_i$ is said to be \emph{bounded} if for every point $x\in\a$, every ray starting at $x$, pointing in the direction opposite the coorientation of $\gamma_i$ toward infinity, passes through an outer arc. In particular, this outer arc is not on $\gamma_i$.
    \end{definition}
    
    \begin{remark}\label{inner_remark}
        We emphasize that for an inner arc to be bounded, there must be a bounding outer arc on a \textit{different} connected component. Recall an inner arc on $\gamma_i$ is comprised of only inner points, and, by Definition \ref{Def:Generalized_Inner_and_Outer}, inner points are those that exist on $\partial U(\gamma_i)$. Therefore, for any inner point $x$ on $\gamma_i$, the ray starting at $x$, pointing in the direction opposite the coorientation of $\gamma_i$ toward infinity, does not intersect any other point of $\gamma_i$. If such an intersection exists, then $x$ is not on $\partial U(\gamma_i)$ and is not an inner point.
    \end{remark}

    \begin{lemma}\label{lemma:inner arcs bounded}
        Suppose $S$ is an orientable compact surface with $k>0$ boundary components immersed into $\R^2$. All inner arcs of the immersed boundary components of $S$ must be bounded.
    \end{lemma}
    \begin{proof}
        Suppose to the contrary that an arc $\a$ of $\gamma_i$ is unbounded. Then, there exists a point $x\in\a$ such that the ray from $x$ to infinity passes through no outer arcs. We then have two cases to consider: either the ray passes through no arcs, or it passes through only inner arcs. If the ray passes through no arcs, then the compactness of $S$ is contradicted since this implies that every point of the unbounded region is the image of a compact surface. If the ray passes through only inner arcs, then consider the last inner arc the ray passes through. We then apply the same argument as before, contradicting the compactness of $S$.
    \end{proof}

    We are now in a position to finish our case of only immersed inner boundaries. Inner curves must have at least one inner arc, and each inner arc must be bounded. Thus, there must be at least two immersed inner boundaries that have at least one outer arc each. We will now show that for an inner curve to have both inner arcs and outer arcs, it must contain at least one negative self-intersection.

    \begin{lemma}\label{2_neg_self_int}
        If $S$ is an orientable compact surface with $k\geq 2$ boundary components immersed into $\R^2$, and all of these immersed boundaries are inner, then there are at least $2$ negative self-intersection points. 
    \end{lemma}
    
    \begin{proof}
        Denote the immersed boundary components of $S$ by $\gamma_i$ for $1 \leq i \leq k$. Since each $\gamma_i$ is an inner component, we know that each $\gamma_i$ must have an inner arc. By Lemma \ref{lemma:inner arcs bounded}, the inner arcs on each $\gamma_i$ must be bounded by at least one outer arc on some $\gamma_j$ for $j\neq i$. Therefore, $\gamma_j$ has at least one inner arc and at least one outer arc. Moreover, as discussed in Remark \ref{inner_remark}, an outer arc on $\gamma_j$ cannot bound the inner arc(s) on $\gamma_j$, therefore there is at least one other $\gamma_l$ for $l \neq j$ with an outer arc. Therefore, $\gamma_l$ has at least one inner arc and at least one outer arc. 
        
        Now, to have both an inner and outer arc on the same curve, there must be a self-intersection point where the curve switches from an inner arc to an outer arc. Figure \ref{fig:NegSelfIntsandOuterArcs_b} shows a neighborhood of a self-intersection point on $\gamma_j$ (respectively, $\gamma_l$) where, by starting at base point $p$ and traversing $\gamma_j$ (respectively, $\gamma_l$) in the direction of its orientation, an inner arc changes into an outer arc. Inherently, by starting at the base point and traversing in the direction of the orientation of $\gamma_j$ (respectively, $\gamma_l)$, this self-intersection point is negative. In conclusion, since both $\gamma_j$ and $\gamma_l$ contain both inner and outer arcs, they each necessarily have a negative self-intersection point. This yields a minimum of two negative self-intersection points for the immersed boundary components of $S$.

    \end{proof}

    \begin{figure}[H]
        \begin{subfigure}{.4\textwidth}
            \centering 
            \begin{tikzpicture}
                %\node (b) at (1,-1) [circle, draw, scale=.3, fill=black] {};
                \node[shape=circle, draw=orange, fill=orange, minimum size=6.5pt, inner sep=0] (b) at (1,-1) {};
                \node[] (fp) at (1, -1.3) {$p$};
                \node (u) at (0, -1.5) {$U$};
    
                %%% Main Arcs %%%
                \draw[->, ultra thick] (b) -- node[near start, sloped, below] {\tiny{Inner}} (-1,1);
                \draw[->, ultra thick] (1,1) -- node[near end, sloped, below] {\tiny{Inner}} (-1,-1);
    
                %%% Coorient %%%
                \draw[->] (.5, -.5) -- (.75,-.25);
                \draw[->] (-.5, .5) -- (-.25,.75);
    
                \draw[->] (.5, .5) -- (.25,.75);
                \draw[->] (-.5, -.5) -- (-.75,-.25);
    
                %%% Dotted %%%
                \draw[dotted, ultra thick] (1,1) arc (0:180:1);
                \clip (-2,-1) rectangle (2,2.6);
                \draw[dotted, ultra thick] (0,.61) circle (1.89);
                           
            \end{tikzpicture}
            \caption{Positive Self-Intersection}
            \label{fig:NegSelfIntsandOuterArcs_a}
        \end{subfigure}
        \begin{subfigure}{.4\textwidth}
            \centering 
            \begin{tikzpicture}
                %\node (b) at (1,-1) [circle, draw, scale=.3, fill=black] {};
                \node[shape=circle, draw=orange, fill=orange, minimum size=6.5pt, inner sep=0] (b) at (1,-1) {};
                \node[] (fp) at (1, -1.3) {$p$};
                \node (u) at (0, -1.5) {$U$};
    
                %%% Main Arcs %%%
                \draw[->, ultra thick] (b) -- node[near start, sloped, below] {\tiny{Inner}} node[near end, sloped, below] {\tiny{Outer}} (-1,1);
                \draw[<-, ultra thick] (1,1) -- node[near end, sloped, above] {\tiny{Outer}} node[near start, sloped, above] {\tiny{Inner}} (-1,-1);
    
                %%% Coorient %%%
                \draw[->] (.5, -.5) -- (.75,-.25);
                \draw[->] (-.5, .5) -- (-.25,.75);
    
                \draw[->] (.5, .5) -- (.75,.25);
                \draw[->] (-.5, -.5) -- (-.25,-.75);
    
                %%% Dotted %%%
                \draw[dotted, ultra thick] (-1,1) arc (90:270:1);
                \draw[dotted, ultra thick] (1,1) arc (90:-90:1);

            \end{tikzpicture}
            \caption{Negative Self-Intersection}
            \label{fig:NegSelfIntsandOuterArcs_b}
        \end{subfigure}
            \caption{Negative Self-Intersections Create Outer Arcs}
            \label{fig:NegSelfIntsandOuterArcs}
        \end{figure}
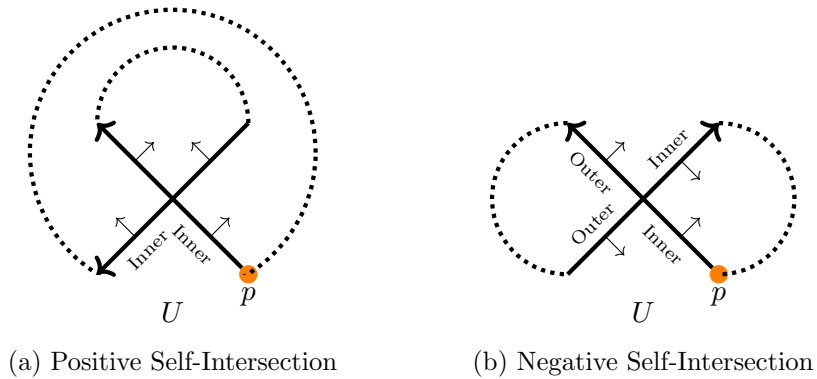

    \begin{lemma}\label{lemma:at least one outer means negative self int}
        Let $f:S \rightarrow \R^2$ be an immersion of a compact, oriented, non-planar surface $S$ with a single immersed outer boundary curve $\gamma_1$ and a collection of inner curves. Then there is a negative self-intersection of one of the components. 
    \end{lemma}

    \begin{proof} 
        By Corollary \ref{corollary: outer with self int}, if $\gamma_1$ contains self-intersections then at least one of them is negative. Thus, we need only consider the case where $\gamma_1$ is embedded without self-intersections. Since $\gamma_1$ is embedded via $f$, by Lemma \ref{guth_extension_lemma_1}, we may positively isotope $\gamma_1$ so that it remains embedded and is disjoint from all other immersed boundary components by some map $f_1:S\to\R^2$. In fact, we can positively isotope $\gamma_1$ so that the remaining inner components are unmodified, and are contained in the bounded region of $\R^2\setminus\gamma_1$.

        Suppose there are $k-1$ inner boundary components $\gamma_2, \ldots, \gamma_k$ and let $U_i$ denote the unbounded region of $\R^2\setminus \gamma_i$. We may select a curve, say $\gamma_2$, a non-double inner point $p_2\in\partial U_2$, and a non-self-intersecting path $\delta\co [0, \infty)\to U_2$ from $p_2$ to $U_1$ chosen in such a way that $\delta$ intersects $\gamma_1$ at a unique point $q_1$ and $\delta$ does not intersect any other $\gamma_i$. We denote the part of $\delta$ from $q_1$ to $p_2$ by $[q_1, p_2]$, and highlight that $f_1|_{f_1^{-1}[q_1, p_2]}$ is a diffeomorphism. Let $S'$ denote the compact surface obtained by removing from $S$ a thin neighborhood of $f_1^{-1}[q_1, p_2]$. Then, the oriented boundary $\partial S'$ is obtained from the oriented boundary $\partial S$ by an oriented surgery, and $f'=f_1|_{S'}$ is an immersion.

        Under $f'$, the immersed boundary of $S'$ has the same number of self-intersections as the immersed boundary of $S$ under $f$, contains one fewer components, and satisfies one of three cases. Either (1) $S'$ has a single immersed outer boundary curve $\gamma_1'$ with self-intersections, (2) $S'$ has a single embedded outer boundary curve $\gamma_1'$, or (3) $S'$ has no immersed outer boundary curves. If $S'$ has no immersed boundary curves, then both the immersed boundaries of $S'$ and $S$ must contain at least two negative self-intersections by Lemma \ref{2_neg_self_int}. If $S'$ has a single immersed outer boundary curve $\gamma_1'$ with self-intersections, then $\gamma_1'$ must contain at least one negative self-intersection by Corollary \ref{corollary: outer with self int}. 
        
        If $S'$ has a single embedded outer boundary curve $\gamma_1'$, then we may once again positively isotope $\gamma_1'$ in such a way that all other components are disjoint from $\gamma_1'$ and in the bounded region of $\R^2\setminus\gamma_1'$. We then repeat the above process by selecting a new inner curve, say $\gamma_2'$, to perform surgery on $\gamma_1'$ to obtain a surface $S''$. Since $S$ is non-planar, after $r$ steps for some $ 1 \leq r \leq k-1$, this process must terminate with either case (1) or case (3).
    \end{proof}

    In summation, we have shown that there must exist a minimum number of negative self-intersections in two separate scenarios. If there is at least one immersed outer boundary, then there is at least one negative self-intersection. If there are no outer boundaries, then there are at least two negative self-intersections. Finally, we can now extend Guth's theorem to (non-planar) surfaces with $k\geq1$ boundary components.

%%%%%%%%%%%%%%%%%%%%%%%%%%%%%%%%%%%%%%%%%%%%%%%%%%%%%%%
%%%%%%%%%%%%%%%%%%%%%%%%%%%%%%%%%%%%%%%%%%%%%%%%%%%%%%%
%%%%%%%%%%%% Below is the 1st Main Theorem %%%%%%%%%%%%
%%%%%%%%%%%%%%%%%%%%%%%%%%%%%%%%%%%%%%%%%%%%%%%%%%%%%%%
%%%%%%%%%%%%%%%%%%%%%%%%%%%%%%%%%%%%%%%%%%%%%%%%%%%%%%%

    \begin{theorem}\label{th:2}\label{Thm:Self_Int_of_Boundary}
    Let $S$ be an oriented, compact, non-planar, not necessarily path-connected, surface with boundary $\partial S$. Suppose that $\partial S$ has $k$ path components. If $f : S \rightarrow \R^2$ is an immersion, then $f(\partial S)$ has at least $4 - k- \chi(S)$ self-intersections. This estimate is sharp. 
    \end{theorem}
    
    Recall that for a compact, oriented, surface $S$ with genus $g$ and $k$ boundary components, its Euler characteristic is $\chi(S)=2-2g-k$. Thus, the previous theorem may be restated to say that $f(\partial S)$ has at least $2g+2$ self-intersections, recovering Guth's theorem when $k=1$. 
    
    \begin{proof} Let $\gamma_1,..., \gamma_k$ be the components of $f(\partial S)$. As defined in Section \ref{TheWhitneyFormula}, let $n^\pm_j=n^\pm(\gamma_j)$, $i^\pm_j=i^{\pm}(\gamma_j)$, and $w_j=w(\gamma_j)$ be the respective invariants of $\gamma_j$ for $j=1,\cdots , k$, and let $N^+$, $N^-$, $I^{+}$, and $I^{-}$ be defined as the sum of the respective invariants. Recall, Equation \ref{Haefliger Result}, the generalized Whitney formula:
        \[\chi(S)= I^+ - I^- + N^- - N^+\]
    Let $\Delta=N^++N^-+N^+_{ij}+N^-_{ij}$ denote the sum of all self-intersections. We see that $\Delta\geq N^++N^-$, since $N^+_{ij}+N^-_{ij}\geq 0$. Thus, we may rewrite the generalized Whitney formula, yielding the following:
        \[\Delta \geq 2N^- + I^+ - I^- - 2+2g+k\]
    We now have three cases to consider. First, if there are no outer components, then there are at least two negative self-intersections, by Lemma \ref{2_neg_self_int}. Thus, $I^+=0$, $I^-=k$, and:
            \begin{align*}
                \Delta
                &\geq 2(2)+0-k-2+2g+k\\
                &=2+2g\\
                &=4-k-\chi(S)
            \end{align*}            
    Second, if there is exactly one outer component, then $N^{-}\geq 1$, by Lemma \ref{lemma:at least one outer means negative self int}. Thus, $I^{+}\geq 1$, $I^{-}\leq k-1$, and:
            \begin{align*}
                \Delta
                &\geq 2(1)+1-(k-1)-2+2g+k\\
                &=2+2g\\
                &=4-k-\chi(S)
            \end{align*}            
    Third, if there are at least two outer components, then $I^+\geq 2$, $I^-\leq k-2$, $N^-\geq 0$, and:
        \begin{align*}
                \Delta
                &\geq 2(0)+2-(k-2)-2+2g+k\\
                &=2+2g\\
                &=4-k-\chi(S)
            \end{align*}
    It remains to show that the estimate is sharp. By the construction of Guth~\cite{Gu09}, the curve $B^+_{2-\chi(S),1}=B^+_{1+2g,1}$ bounds a surface $S$ of genus $g$ with a single path-connected boundary $(k=1)$. The curve $B^+_{1+2g,1}$ has $1+2g$ positive self-intersections and $1$ negative self-intersection, and so $\Delta=2+2g$. In order to generalize from a single boundary component to $k\geq 2$, we need only cut $k-1$ disjoint discs out of the immersed surface from the interior of $B^+_{1+2g,1}$.

\end{proof}

%%%%%%%%%%%%%%%%%%%%%%%%%%%%%%%%%%%%%%%%%%%%%%%%%%%%%%%
%%%%%%%%%%%%%%%%%%%%%%%%%%%%%%%%%%%%%%%%%%%%%%%%%%%%%%%
%%%%%%%%%%%% Above is the 1st Main Theorem %%%%%%%%%%%%
%%%%%%%%%%%%%%%%%%%%%%%%%%%%%%%%%%%%%%%%%%%%%%%%%%%%%%%
%%%%%%%%%%%%%%%%%%%%%%%%%%%%%%%%%%%%%%%%%%%%%%%%%%%%%%%

    \begin{remark}
        We provide an alternative way to prove the sharpness condition of the lower bound from Theorem \ref{Thm:Self_Int_of_Boundary} using positive concordances guaranteed to exist by Lemma \ref{guth_extension_lemma_1} and Lemma \ref{guth_extension_lemma_2}. Let $S'$ be a genus $g$ surface with one boundary component. By the construction of Guth~\cite{Gu09}, the curve $B^+_{2-\chi(S'),1}$ bounds the immersed surface $S'$. Let $S$ be a surface obtained from $S'$ by attaching $k-1$ handles of index $1$ in such a way that $\partial S$ has $k$ path components. Then $\chi(S)=\chi(S')-k+1$.  Recall, by Example \ref{splitting_Bs_example},  $B^+_{2-\chi(S'),1}$ is positively concordant to the union of curves $B^+_{i_0,1}$ and $C^+_{i_1}\sqcup \cdots \sqcup C^+_{i_s}$ where $2-\chi(S')=i_0+\cdots + i_s$. Consequently, the curves $B^+_{i_0,1}$ and $C^+_{i_1}\sqcup \cdots \sqcup C^+_{i_s}$ bound the immersed surface $S$. The number of intersection points of the curves $B^+_{i_0,1}$ and $C^+_{i_1}\sqcup \cdots \sqcup C^+_{i_s}$ is 
            \[
                  (1+i_0)+i_1+\cdots + i_s = 1 + 2-\chi' =3- (\chi(S)+k-1)=4-k-\chi(S),
            \]
        which is precisely the lower estimate for the number of self-intersection points given by Theorem~\ref{th:2}.
    \end{remark}
   \iffalse    
    \begin{corollary}
        Let $F$ and $S$ be surfaces with boundary of genus $g_F$ and $g_S$, respectively, and denote the number of boundary components of $F$ and $S$ by $k_F$ and $k_S$, respectively. Let $F\#S$ be the usual connect-sum operation, and consider maps $f:F \rightarrow \R^2$, $g:S \rightarrow \R^2$, and $f \# g: F \# S \rightarrow \R^2$ such that $f$, $g$, and $f \# g$ are immersions on the respective boundaries. Then, the number of self-intersections of $(f\#g)(\partial (F \# S))$ is at least $2(g_F + g_S) + 2$.
    \end{corollary}

    \begin{proof}
        Let $k=k_F+k_S$ denote the number of boundary components of $F \# S$ and let $\Delta_{\#}$ denote the number of self-intersections of $(f\#g)(\partial (F \# S))$. By Theorem \ref{th:2}, we have that:
            \[\Delta_{\#}\geq 4 - k - \chi(F \# S)\]
        Finally, since $\chi(F \# S)=\chi(F)+\chi(S)-2$, we have the following:
        \begin{align*}
            4 - k - \chi(F \# S) 
                &=  4 - (k_F + k_S)  - (\chi(F) + \chi(S) - 2)  \\
                &= 4 - (k_F + k_S) - ( (2 - 2g_F - k_F) + (2 - 2g_S - k_S) - 2)\\
                %&= 4 - (k_F + k_S) - (2 - 2g_F - 2g_S - k_F - k_S) \\
                &= 2(g_F + g_S) + 2
        \end{align*}        
    \end{proof}
\fi
    \begin{example}[Genus 0]  
        If $f\co S\to \R^2$ be the standard inclusion of a unit disc, then $4-k-\chi(S)=2>0$.  Thus, in this case the estimate in Theorem \ref{Thm:Self_Int_of_Boundary} is indeed not applicable. 
    \end{example}
    
    \begin{example}[One Boundary Component] 
        Suppose that $k=1$. Then the number of self-intersection points is $4-1-\chi(S)=2g+2$, which recovers Theorem~\ref{th:Gu}. 
    \end{example}

    As highlighted in Remark \ref{B curves are important} and Theorem \ref{Thm:Self_Int_of_Boundary}, the sharpness of Guth's result is a surface of genus $g$ with immersed boundary represented by the curve $B^+_{1+2g,1}$. For surfaces with multiple boundary components, the following proposition highlights a possible configuration for such an immersed boundary.

    \begin{proposition}\label{Prop:SpecificBoundaries} Let $S_{g, k}$ be a surface of genus $g$ with $k$ boundary components. There is an immersions $f:S_{g, k}\to\R^2$ such that:
        \[f(\partial S_{g, k})=B^+_{i_0,1}\sqcup (C^+_{i_1}\sqcup \cdots\sqcup C^+_{i_{k-1}})\]
    where $1+\sum i_j=2+2g$.
    \end{proposition}
    \begin{proof}    
        Guth \cite{Gu09} showed that for any surface $S_{g,1}$ of genus $g$ with a single boundary component that there is an immersion of $S_{g,1}$ such the $f(\partial S_{g,1})=B_{2g+1,1}^+$ has exactly $2g+2$ self-intersections. By Example \ref{splitting_Bs_example},  for any decomposition of $2g+1=i_0+\cdots+i_{k-1}$ into positive integers we have that:
            \[B_{2g+1,1}^+=B_{i_{0},1}^+\sqcup \left(C_{i_{1}}^+\sqcup\cdots\sqcup C_{i_{k-1}}^+\right)\]
        Thus, we have that the immersed boundary of $S_{g,k}$ is now:
            \[f(\partial S_{g,k})=B^+_{i_0,1}\sqcup (C^+_{i_1}\sqcup\cdots\sqcup C^+_{i_{k-1}})\]
    \end{proof}

\section{Weighted bipartite graphs associated with fold maps}\label{section:SplittingAndGraphConstruction}

    In this section, we introduce notions that are necessary to state Theorem \ref{Thm:Main_Theorem_2025} and Theorem \ref{thm:main theorem sharpness}. We begin by defining a splitting of a surface. Let $f\co S\to \R^2$ be a simple fold map of an oriented closed surface $S$.\label{Graph_Construction} Let $S_+$ (respectively, $S_-$) denote the closure in $S$ of the set of points $x\in S$ such that $d_xf$ is a linear map preserving (respectively, reversing) orientation. Then $S=S_+\cup S_-$, and the set $\Sigma=S_+\cap S_-=\partial S_+=\partial S_-$ consists of singular points of $f$. 
    
     \begin{definition}[Splitting]\label{definition: splitting}
    Let $S_{+}^i$ denote path components of the surface $S_+$ and let $S_{-}^j$ denote path components of the surface $S_-$.  Denote the intersections of path components $S_{+}^i\cap S_{-}^j$ by $(\Sigma_{j}^i)_m$, as shown in Figure \ref{fig:GraphOfSurface_a}. We call the decomposition of $S$, with respect to $f$, into the collections $S_+=\set{S_{+}^i}$ and $S_-=\set{S_{-}^j}$ the \emph{splitting} of $S$.
    \end{definition}

    Surfaces with a prescribed splitting can be viewed as graphs. More precisely, we associate a bipartite graph $G$ with $f$. The vertices of $G$, denoted by $v_{+}^{i}$ and $v_{-}^{j}$, correspond to the interiors of the path components $S_{+}^i$ and $S_{-}^j$. An edge $e_{i}^j$ in $G$ between two vertices $v_{+}^i$ and $v_{-}^j$ exists if and only if $S_{+}^i\cap S_{-}^j$ is non-empty. We say that the \emph{multiplicity} of an edge $e_{i}^j$, denoted by $|e_{i}^j|$, is $m$ if $(\Sigma_{j}^i)_m$ consists of $m$ path components. If the multiplicity of an edge is more than 1, we will use $(e_i^j)_m$ to distinguish them; otherwise we will suppress the $m$ for notation-sake.

    \begin{definition}\label{def:rhoG}

        A \emph{spanning tree} of a connected graph $G$ is a connected subgraph containing all vertices of $G$ and no cycles. Define $\rho(G)$ to be the number of edges of a spanning tree of $G$ associated to a splitting induced by a simple fold map $f$. In particular, if $G$ has $k$ vertices, then $\rho(G)=k-1$.

    \end{definition}

    \begin{example}
        Figure \ref{fig:GraphOfSurface_a} depicts a surface $S$ with splitting $S_{+}$ and $S_-$ in blue and orange, respectively. (Recall, this splitting is determined by a map $f : S\to\R^2$.) The path components of the splitting are also indicated. The corresponding graph can be seen in Figure \ref{fig:GraphOfSurface_b}. The path components of $S_+$ and $S_-$ are represented as vertices $v_+^i$ and $v_-^j$, respectively, and the edges $e_i^j$ represent the boundary path components in the splitting. In this example, we note the multiplicity of $e_2^1$ is $3$ and that the number of edges in a spanning tree of the graph is $3$. 
    \end{example}

    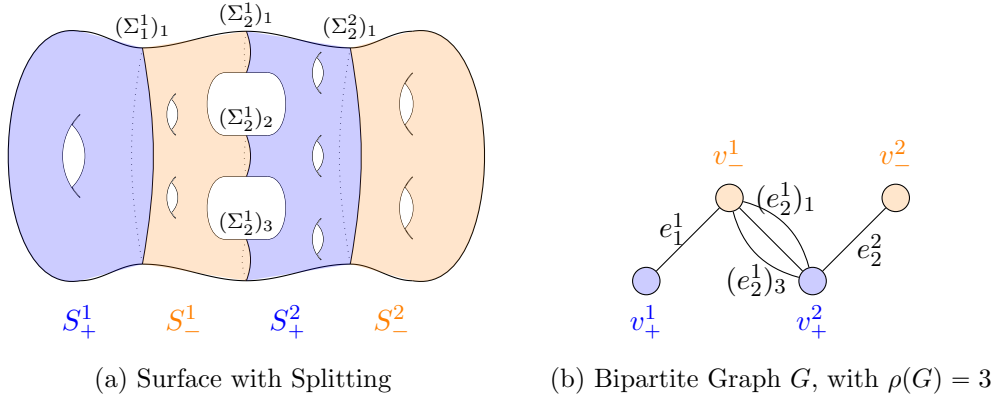
\begin{figure}[H]
    \begin{subfigure}{0.45\textwidth}
        \centering 
            \begin{tikzpicture}[scale=0.55]
                %%% 1st Upper Torus Line
                \draw[name path=U1, domain=0:1.5, smooth, variable=\x] plot ({\x}, {.2*cos(deg(2*3.14159/3*\x))+5.8} );

                %%% 1st Lower Torus Line
                \draw[name path=L1, domain=0:1.5, smooth, variable=\x] plot ({\x}, {-.2*cos(deg(2*3.14159/3*\x))+.2} );
                
                %%% Left-most Torus
                \draw[name path=A] (0,6) to [bend right=100] (0, 0);
                \draw[name path=B] (0,4) to [bend right=45] (0,2);
                \draw[name path=C] (-.2,3.75) to [bend left=45] (-.2,2.25);

                %%% Border Line b/w 1 and 2
                \draw[name path=D, very thin] (1.5,5.6) to [bend left=10] (1.5, 0.4);
                \draw[very thin, dotted] (1.5,5.6) to [bend right=10] (1.5, 0.4);

                %%% 2nd Torus
                %\draw (0,6) to [bend right=30] (3,6);
                %\draw (0,0) to [bend left=30] (3,0);
                \draw[name path = T2UL] (4.5,5) to (3.5,5) to [bend right=90] (3.5,3.5) to (4.5,3.5);
                \draw[name path = T2LL] (4.5,2.5) to (3.5,2.5) to [bend right=90] (3.5,1) to (4.5,1);
                %\draw (3,5) arc (90:270:.75);

                \draw[name path = T2UR] (4.5,5) to [bend left=90] (4.5,3.5);
                \draw[name path = T2LR] (4.5,2.5) to [bend left=90] (4.5,1);

                %%% 2nd Upper Torus Line
                \draw[domain=1.5:6.5, smooth, variable=\x] plot ({\x}, {-.2*cos(deg(2*3.14159/5*(\x-1.5)))+5.8} );

                %%% 2nd Lower Torus Line
                \draw[domain=1.5:6.5, smooth, variable=\x] plot ({\x}, {.2*cos(deg(2*3.14159/5*(\x-1.5)))+.2} );

                 %%% Border Lines b/w 2 and 3
                \draw[name path=B1, very thin] (4,6) to [bend left=20] (4, 5);
                \draw[very thin, dotted] (4,6) to [bend right=10] (4,5);

                \draw[name path=B2, very thin] (4,3.5) to [bend left=20] (4, 2.5);
                \draw[very thin, dotted] (4,3.5) to [bend right=10] (4,2.5);

                \draw[name path=B3, very thin] (4,1) to [bend left=20] (4, 0);
                \draw[very thin, dotted] (4,1) to [bend right=20] (4,0);

                %%% Final Torus
                %% Upper
                \draw[domain=6.5:8, smooth, variable=\x] plot ({\x}, {-.2*cos(deg(2*3.14159/3*(\x-6.5)))+5.8} );
                %% Lower
                \draw[domain=6.5:8, smooth, variable=\x] plot ({\x}, {.2*cos(deg(2*3.14159/3*(\x-6.5)))+.2} );    

                %% Genus in far-right Orange
                \draw[name path=Z] (8,6) to [bend left=100] (8, 0);
                \draw[name path=Z1] (8,5) to [bend right=45] (8,3.5);
                \draw[name path=Z2] (7.8,4.75) to [bend left=45] (7.8,3.75);
                \draw[name path=Z3] (8,2.5) to [bend right=45] +(0,-1.5);
                \draw[name path=Z4] (7.8,2.25) to [bend left=45] +(0,-1);

                %%% Border Lines b/w 3 and 4
                \draw[name path=D3, very thin] (6.5,5.6) to [bend left=10] (6.5, 0.4);
                \draw[very thin, dotted] (6.5,5.6) to [bend right=10] (6.5, 0.4);

                %%% Fillings
                \path[name path=D2] (4,6) to [bend left=20] (4, 5) -- (4,3.5) to [bend left=20] (4, 2.5) -- (4,1) to [bend left=20] (4, 0);
                \tikzfillbetween[of=A and D]{blue, opacity=0.2};
                \tikzfillbetween[of=B and C]{white, opacity=1};

                \tikzfillbetween[of=D and D2]{orange, opacity=0.2};
                \tikzfillbetween[of=D2 and D3]{blue, opacity=0.2};
                \tikzfillbetween[of=T2UL and T2UR]{white, opacity=1};
                \tikzfillbetween[of=T2LL and T2LR]{white, opacity=1};

                \tikzfillbetween[of=D3 and Z]{orange, opacity=0.2};
                \tikzfillbetween[of=Z1 and Z2]{white, opacity=1};
                \tikzfillbetween[of=Z3 and Z4]{white, opacity=1};

                %%%% Labels 
                \node at (0,-1) {\color{blue}$S_{+}^{1}$};
                \node at (2.5,-1) {\color{orange}$S_{-}^{1}$};
                \node at (5,-1) {\color{blue}$S_{+}^{2}$};
                \node at (7.5,-1) {\color{orange}$S_{-}^{2}$};

                %%%% Sigmas
                \node at (1.5, 6.1) {\scriptsize$(\Sigma_1^1)_1$};

                \node at (4, 6.4) {\scriptsize$(\Sigma_2^1)_1$};
                \node at (4, 3.9) {\scriptsize$(\Sigma_2^1)_2$};
                \node at (4, 1.4) {\scriptsize$(\Sigma_2^1)_3$};

                \node at (6.5, 6.1) {\scriptsize$(\Sigma_2^2)_1$};

                %%%% Extra Genus Added Later
                % Left-most orange
                \draw[name path=NLL] (2.3,4.5) to [bend right=45] (2.3,3.5);
                \draw[name path=NLS] (2.2,4.35) to [bend left=45] (2.2,3.65);
                \tikzfillbetween[of=NLL and NLS]{white, opacity=1};

                \draw[name path=NLL2] (2.3,2.5) to [bend right=45] (2.3,1.5);
                \draw[name path=NLS2] (2.2,2.35) to [bend left=45] (2.2,1.65);
                \tikzfillbetween[of=NLL2 and NLS2]{white, opacity=1};

                % right-most blue
                \draw[name path=NRL1] (5.8,5.5) to [bend right=45] (5.8,4.5);
                \draw[name path=NRS1] (5.7,5.35) to [bend left=45] (5.7,4.65);
                \tikzfillbetween[of=NRL1 and NRS1]{white, opacity=1};

                \draw[name path=NRL2] (5.8,3.5) to [bend right=45] (5.8,2.5);
                \draw[name path=NRS2] (5.7,3.35) to [bend left=45] (5.7,2.65);
                \tikzfillbetween[of=NRL2 and NRS2]{white, opacity=1};

                \draw[name path=NRL3] (5.8,1.5) to [bend right=45] (5.8,0.5);
                \draw[name path=NRS3] (5.7,1.35) to [bend left=45] (5.7,0.65);
                \tikzfillbetween[of=NRL3 and NRS3]{white, opacity=1};

                %\draw (0,6) to [bend left=45] (0, 0);
            \end{tikzpicture}
        \caption{Surface with Splitting}
        \label{fig:GraphOfSurface_a}
    \end{subfigure}
    \begin{subfigure}{0.45\textwidth}
        \centering 
        \begin{tikzpicture}[scale=1.1]
            \node[label=below:{\color{blue}$v_+^1$}] (a) at (0,0) [circle, draw, minimum size=2pt, fill=blue!20] {};
            \node[label=above:{\color{orange}$v_-^1$}] (b) at (1,1) [circle, draw, minimum size=2pt, fill=orange!20] {};
            \node[label=below:{\color{blue}$v_+^2$}] (c) at (2,0) [circle, draw, minimum size=2pt, fill=blue!20] {};
            \node[label=above:{\color{orange}$v_-^2$}] (d) at (3,1) [circle, draw, minimum size=2pt, fill=orange!20] {};

            \draw (a) -- node [above, near start] {$e_1^1$} (b);
            
            \draw (b) --  (c);
            \draw (b) to [bend right=30] node [below] {$(e_2^1)_3$}  (c);
            \draw (b) to [bend left=30] node [above] {$(e_2^1)_1$} (c);

            \draw (c) -- node [below, near end] {$e_2^2$} (d);

        \end{tikzpicture}
        \caption{Bipartite Graph $G$, with $\rho(G)=3$}
        \label{fig:GraphOfSurface_b}
    \end{subfigure}
        \caption{Graph Corresponding to Surface}
        \label{fig:GraphOfSurface}
    \end{figure}

    Recall that $(\Sigma_i^j)_{m}$ denotes path components of the intersection $S_+^i\cap S_-^j$, and the collection of all $(\Sigma_i^j)_{m}$ is the singular set of $f$. For notational clarity, we will often omit decorations and refer to the collection as simply $\Sigma$. We now wish to endow each of the closed curves in the collection $f(\Sigma)$ with a coorientation and orientation. First, choose a basepoint on each path component of $f(\Sigma)$. Second, define the coorientation of each component of $f(\Sigma)$ in the following way. If the the number of points in the inverse image of a regular value of $f|_{S_{+}^i}$ is smaller than the number of points in the inverse image of a regular value of $f|_{S_-^j}$, then the coorientation on $f(\Sigma_{i}^j)$ points towards $f|_{S_+^i}$. Otherwise, if the the number of points in the inverse image of a regular value of $f|_{S_{-}^j}$ is smaller than the number of points in the inverse image of a regular value of $f|_{S_+^i}$, then the coorientation on $f(\Sigma_{i}^j)$ points towards $f|_{S_-^j}$. Finally, orient each $f(\Sigma_i^j)$ so that the coorientation followed by the orientation on $f(\Sigma_i^j)$ agrees with the usual orientation on $\R^2$.

    \begin{remark}[Weights on the Graph]
        A simple fold map $f\co S\to \R^2$ defines weights on its corresponding graph $G$. Namely, the weight of a vertex $v_{+}^i$ corresponding to a path component $S_{+}^i$ is $\chi(S_{+}^i)$. Denote the weight on a vertex by $\chi(v_{+}^i)$. An analogous weight $\chi(v_-^j)$ is defined for path components of $S_{-}^j$. The weight of an edge $(e_{i}^j)_m$ corresponding to a path component $(\Sigma_{i}^{j})_m$ is a quadruple 
        \[(i_+((e_{i}^j)_m),i_-((e_{i}^j)_m),n_+((e_{i}^j)_m), n_-((e_{i}^j)_m)).\] 
    \end{remark}
    
    We now aim to analyze the number of transverse self-intersections of $f(\Sigma)$, denoted by $\Delta_{\Sigma}$. Also, denote the number of path components of $\Sigma$, $S_+$, and $S_-$ by $|\Sigma|$, $\#|S_+|$ and $\#|S_-|$ respectively. 

    \begin{lemma}\label{l:split}
        If $f\co S\to \R^2$ is a simple fold map of an oriented closed surface with a splitting $S=S_+\cup S_-$ as above, then the following are true:
            \begin{enumerate}[itemsep=0.1in, label=(\roman*)]
                \item $\chi(S)=\chi(S_+)+\chi(S_-)$
                \item $\chi(S_+)=\chi(S_-)$
                \item $\#|S_+|-\#|S_-|=\displaystyle\sum_{i}g_+^i-\displaystyle\sum_{j}g_-^j$
            \end{enumerate}
        where $g_+^i$, respectively $g_-^j$, is the genus of each path component $S_+^i$, respectively $S_-^j$, and $1\leq i\leq \#|S_+|$ and $1\leq j\leq \#|S_-|$. 
    \end{lemma}
    \begin{proof}
        Part ($i$) is immediate since $\chi(S)=\chi(S_+) + \chi(S_-) - \chi(\Sigma)$ and $\Sigma$ is a collection of circles. 
        
        Part $(ii)$ is the application of a result of Èliašberg \cite[Theorem B]{El70} to our specific case and constructions. Èliašberg proved that for a one-dimensional submanifold $V$ of a closed oriented surface $S$, there exists an immersion $f:S \rightarrow \R^2$ such that $\Sigma = V$ if and only if $V$ splits $S$ into two submanifolds $S_1$ and $S_2$, such that $\partial S_1 = \partial S_2$ and $\chi(S_1) = \chi(S_2)$. In our context, we have an immersion $f:S \rightarrow \R^2$  with singular set $\Sigma$ if and only if $\Sigma =\partial S_+ = \partial S_-$ and $\chi(S_+) = \chi(S_-)$.

        Part $(iii)$ is a consequence of part $(ii)$. Indeed, since $\chi(S_+)=\chi(S_-)$ we see:
            \begin{align*}
                \chi(S_+)&=\chi(S_-)\\[0.1in]
                \Rightarrow \sum_i\chi(S_+^i)&=\sum_j\chi(S_-^j)\\
                \Rightarrow 2(\#|S_+|)-|\Sigma|-\sum_{i}2g_+^i &=2(\#|S_-|)-|\Sigma|-\sum_{j}2g_-^j \\
                % \#|S_+|-\sum_{i}g_+^i &=\#|S_-|-\sum_{j}g_-^j\\
                \Rightarrow \#|S_+|-\#|S_-|&=\sum_{i}g_+^i-\sum_{j}g_-^j
            \end{align*}

    \end{proof}

   To construct a minimum for $\Delta_{\Sigma}$, we now examine the path components $\{S_+^i\}$ and $\{S_-^j\}$ of the splitting. Note, each path component is an oriented surface with boundary immersed via $f$ into $\R^2$. Thus, using methods from Section \ref{extending_guth}, we may examine the count, $\Delta_{+}^{i}$, respectively $\Delta_{-}^{j}$, of self-intersections of the boundaries of each path component $f(S^i_+)$, respectively $f(S^j_-)$.

    \begin{lemma}\label{connecting_sing_to_boundaries}
        If $f\co S\to \R^2$ is a simple fold map with splitting $S=S_+\cup S_-$ and fold singular set $\Sigma=S_+\cap S_-$, then the number of self-intersections of the fold singularities is equal to the sum of all self-intersections of the boundaries divided by two.
            \[\Delta_{\Sigma}=\sum_{i}\Delta_+^i=\sum_{j}\Delta_-^j=\frac{1}{2}\left(\sum_{i}\Delta_+^i+\sum_{j}\Delta_-^j\right)\]
    \end{lemma}
    \begin{proof}
        Let us first observe that the number of self-intersections of $f(\Sigma)$, denoted by $\Delta_{\Sigma}$, is equal to the sum of the self-intersections of each of the path components of $f(\Sigma)$. Each path component of $f(\Sigma)$ is represented by an edge $(e_i^j)_m$ in the graph $G$. Therefore, $\Delta_{\Sigma}$ is the sum of $n_+((e_{i}^j)_m)$ and $n_-((e_{i}^j)_m)$ across all edges:
            \[\Delta_{\Sigma}=\sum_{i,j,m}n_+((e_{i}^j)_m)+n_-((e_{i}^j)_m)\]
        
        Second, we examine the self-intersections of boundaries. If two path components share a boundary, i.e. $S_+\cap S_-\neq\emptyset$, we call them \textit{adjacent} path components.  Given two adjacent path components $S_+^i$ and $S_-^j$ denote $\partial_i^j$ as the sum of all self-intersections of all boundary path components between $S_+^i$ and $S_-^j$.
            \[\partial_i^j=\sum_{m=1}^{|e_i^j|} n_+((e_{i}^j)_m)+ n_-((e_{i}^j)_m)\]
        
        Recall that $\Delta_+^i$ represents the number of self-intersections of \textbf{all} boundaries on the path component $S_+^i$. For a fixed $i$ let $J$ be an indexing set representing all path components $S_-^j$ that are adjacent to $S_+^i$. Summing across all boundaries between all adjacent path components, we have that:
            \[\Delta_+^i=\sum_{j\in J}\partial_i^j\]
        A analogous equation can also be made for $\Delta_-^j$. From here, it is clear to see that by summing across all $\Delta_+^i$ we have calculated $\Delta_{\Sigma}$.
    \end{proof}

    The following lemma extends the results of Theorem \ref{Thm:Self_Int_of_Boundary} to the self-intersections of fold singularities.

    \begin{lemma}\label{Thm:Self_int_of_folds_are_max}
         If $f\co S\to \R^2$ is a simple fold map with splitting $S=S_+\cup S_-$ that has no planar components whose fold singular set is $\Sigma=S_+\cap\, S_-$, then the number of self-intersections of fold singularities has a lower bound given by
            \[\Delta_{\Sigma}\geq\max\set{\sum_{i}2g_+^i+2,\sum_{j}2g_-^j+2}\]
        where $g_+^i$, respectively $g_-^j$, is the genus of each path component $S_+^i$, respectively $S_-^j$.
    \end{lemma}
    \begin{proof}
        The path components of the splitting $S^{i}_{+}$ and $S^{j}_{-}$ are each oriented surfaces with boundary immersed via $f$ into $\R^2$. Recall, $\Delta_+^i$, respectively $\Delta_-^j$, denotes the count of self-intersections of the boundaries of $S^{i}_{+}$, respectively $S^{j}_{-}$. By Theorem \ref{Thm:Self_Int_of_Boundary}, we have a lower bound for the number of self-intersections of the boundaries of $S^{i}_{+}$ and $S^{j}_{-}$:
            \[\Delta_+^i\geq 2g_+^i+2\qquad\text{and}\qquad \Delta_-^j\geq 2g_-^j+2\]
        Utilizing this fact and Lemma \ref{connecting_sing_to_boundaries}, we see:
            \begin{align*}
                \Delta_{\Sigma}
                &=\sum_{i}\Delta_+^i   & \Delta_{\Sigma}&=\sum_{j}\Delta_-^j\\
                &\geq \sum_{i}2g_+^i+2 & &\geq \sum_{j}2g_-^j+2
            \end{align*}
        Thus, $\Delta_{\Sigma}$ must be greater than or equal to both of its lower bounds. Hence, we can clearly see that:
            \[\Delta_{\Sigma}\geq\max\set{\sum_{i}2g_+^i+2,\sum_{j}2g_-^j+2}\]
    \end{proof}

\section{The Main Theorems}\label{section:MainTheorems}

    We are now equipped to state and prove the main results of the paper. 

    \begin{theorem}\label{Thm:Main_Theorem_2025}
        Let $f\co S\to \R^2$ be a simple fold map with splitting $S=S_+\cup S_-$ that has no planar components whose fold singular set is $\Sigma=S_+\cap\, S_-$. If $\abs{\#|S_+|-\#|S_-|}=n$, then the number of self-intersections of fold singularities
            \begin{align*}
                \Delta_{\Sigma}&\geq 4\max\set{\#|S_+|,\#|S_-|}-(\chi(S)/2+|\Sigma|)\\[0.1in]
                &=2(\rho(G)+1+n)-(\chi(S)/2+|\Sigma|)
            \end{align*}
        where $|\Sigma|$ is the number of path components of fold-singularities, and $\rho(G)$ is the number of edges in a spanning tree of the graph $G$ corresponding to $f$.
    \end{theorem}
    
    \begin{proof}
        Recall, each path component of the splitting $S_+^i$, respectively $S_-^j$, is an immersed surface with boundary. Let $g_+^i$, $\chi(S_+^i)$, and $k_+^i$ denote the genus, Euler characteristic, and number of boundary components of $S_+^i$ respectively, and let $g_-^j$, $\chi(S_-^j)$, and $k_-^j$ denote the analogous properties of $S_-^j$. We note that $\sum_{i}k_+^i=\sum_{j}k_-^j=|\Sigma|$, since these boundaries are the fold singularities. Furthermore, note that for each path component $S_+^i$, respectively $S_-^j$, we have $2g_+^i+2=4-k_+^i-\chi(S_+^i)$, respectively $2g_-^j+2=4-k_-^j-\chi(S_-^j)$. Also, by Lemma \ref{l:split}, $\chi(S_+)=\chi(S_-)$ and $\chi(S)=\chi(S_+)+\chi(S_-)$. Thus, Lemma \ref{Thm:Self_int_of_folds_are_max} may be restated as
            \begin{align*}
                \Delta_{\Sigma}
                &\geq\max\set{\sum_{i}2g_+^i+2,\sum_{j}2g_-^j+2}\\
                &=\max\set{\sum_{i}(4-k_+^i-\chi(S_+^i)),\sum_{j}(4-k_-^j-\chi(S_-^j))}\\
                &=\max\set{4(\#|S_+|)-|\Sigma|-\chi(S_+),4(\#|S_-|)-|\Sigma|-\chi(S_-)}\\
                &=4\max\set{\#|S_+|,\#|S_-|}-|\Sigma|-\frac{1}{2}\chi(S)
            \end{align*}
    
        Finally, recall that $\rho(G)$ is the number of edges in a spanning tree of the graph $G$ corresponding to $f$ and, for bipartite graphs associated with the splitting of a surface, is one less than the total number of path components of the splitting; that is $\rho(G)=\#|S_+|+\#|S_-|-1$. Therefore, if $\abs{\#|S_+|-\#|S_-|}=n$ we can see that $2(\rho(G)+1+n)=4\max\set{\#|S_+|,\#|S_-|}$. Thus, we conclude
            \[\Delta_{\Sigma}\geq 2(\rho(G)+1+n)-(\chi(S)/2+|\Sigma|)\]
    \end{proof}

    \begin{corollary}
        If $f\co S\to \R^2$ is a simple fold map with splitting $S=S_+\cup S_-$ that has no planar components whose fold singular set is $\Sigma=S_+\cap\, S_-$, such that $g=\rho(G) + 1$  and $\#|S_+|=\#|S_-|$, then the number of self-intersections of fold singularities $\Delta_{\Sigma} \geq 2g$.
            % \[
            % \Delta_{\Sigma} \geq 2g %=3g - (|\Sigma| + 1)
            % \]
    \end{corollary}
\begin{proof}
        Suppose the surface $S$ has genus $g$. (Note, $g\geq 2$ or we will have a planar component of the splitting.) By Theorem \ref{Thm:Main_Theorem_2025}, we see:
            \[\Delta_{\Sigma}\geq 2(\rho(G)+1+n)-(\chi(S)/2+|\Sigma|)= 2(g+0)-((1-g)+(g-1))=2g.\]
    \end{proof}

    \begin{remark}\label{remark:admissible combo}
        Given a simple fold map of a surface $S$ with no planar components, there are numerous limitations on the genus $g$, the fold singular set $\Sigma$, and the path components of the splitting $S=S_+\cup S_-$. These limitations are largely combinatorial and due to both how fold maps induce splittings on surfaces and our restriction to splittings with no planar components. Given the construction of splittings in Section \ref{section:SplittingAndGraphConstruction} and that $\Sigma=S_+\cap S_-$, the proof of Lemma \ref{lemma:limitations on simple fold maps} for parts $(i)$ through $(iii)$ is apparent. We do emphasize that part $(iv)$ is highly nontrivial. Condition $(iv)$ was proven by K\'alm\'an in \cite[Proposition 1.7]{Kalman_2000}, and it is included here for completeness.
    \end{remark}

    \begin{lemma}\label{lemma:limitations on simple fold maps}
        Let $S$ be a closed, oriented surface of genus $g$. If $f\co S\to \R^2$ is a simple fold map with splitting $S=S_+\cup S_-$ that has no planar components whose fold singular set, denoted by $\Sigma=S_+\cap S_-$, is nonempty, then the following are true:
            \begin{enumerate}[label=(\roman*),itemsep=0.15cm]
                \item $g\geq 2$
                \item $g>|\Sigma|\geq1$
                \item $|\Sigma|\geq\#|S_+|+\#|S_-|-1 \geq1$
                \item $|\Sigma|-g\equiv 1\mod 2$
            \end{enumerate}
    \end{lemma}

    \begin{definition}
        We call any collection of integers $g$, $|\Sigma|$, $\#|S_+|$, and $\#|S_-|$ satisfying the conclusion of Lemma \ref{lemma:limitations on simple fold maps}, an \emph{admissible combination}. 
    \end{definition}

    \begin{theorem}\label{thm:main theorem sharpness}
        Let $g$, $|\Sigma|$, $\#|S_+|$, and $\#|S_-|$ be an admissible combination. If $S$ is a closed, oriented surface with genus $g$, then there exists a simple fold map $f:S\to\R^2$ with splitting $S=S_+\cup S_-$, that has no planar components whose fold singular set is $\Sigma=S_+\cap S_-$, such that the number of self-intersections of fold singularities is
            \[\Delta_{\Sigma}= 2(\rho(G)+1+n)-(\chi(S)/2+|\Sigma|)\]
        where $|\Sigma|$ is the number of path components of fold-singularities, $\rho(G)$ is the number of edges in a spanning tree of the graph $G$ corresponding to $f$, and $\abs{\#|S_+|-\#|S_-|}=n$.
    \end{theorem}
    \begin{proof}
        Suppose without loss of generality that $\#|S_+|\geq\#|S_-|$. We begin by constructing the surface $S$ by examining the surface path components and how they are glued together to form $S$. Recall, the set of folds of the simple fold map $f:S\to\R^2$ consists of the intersections of these components. We now construct the surface $S$. The desired surface $S$ is obtained by constructing three sub-surfaces: the so-called spine, the skull, and the teeth. 

        \begin{figure}[h]
            \centering
            \includegraphics[width=0.5\linewidth]{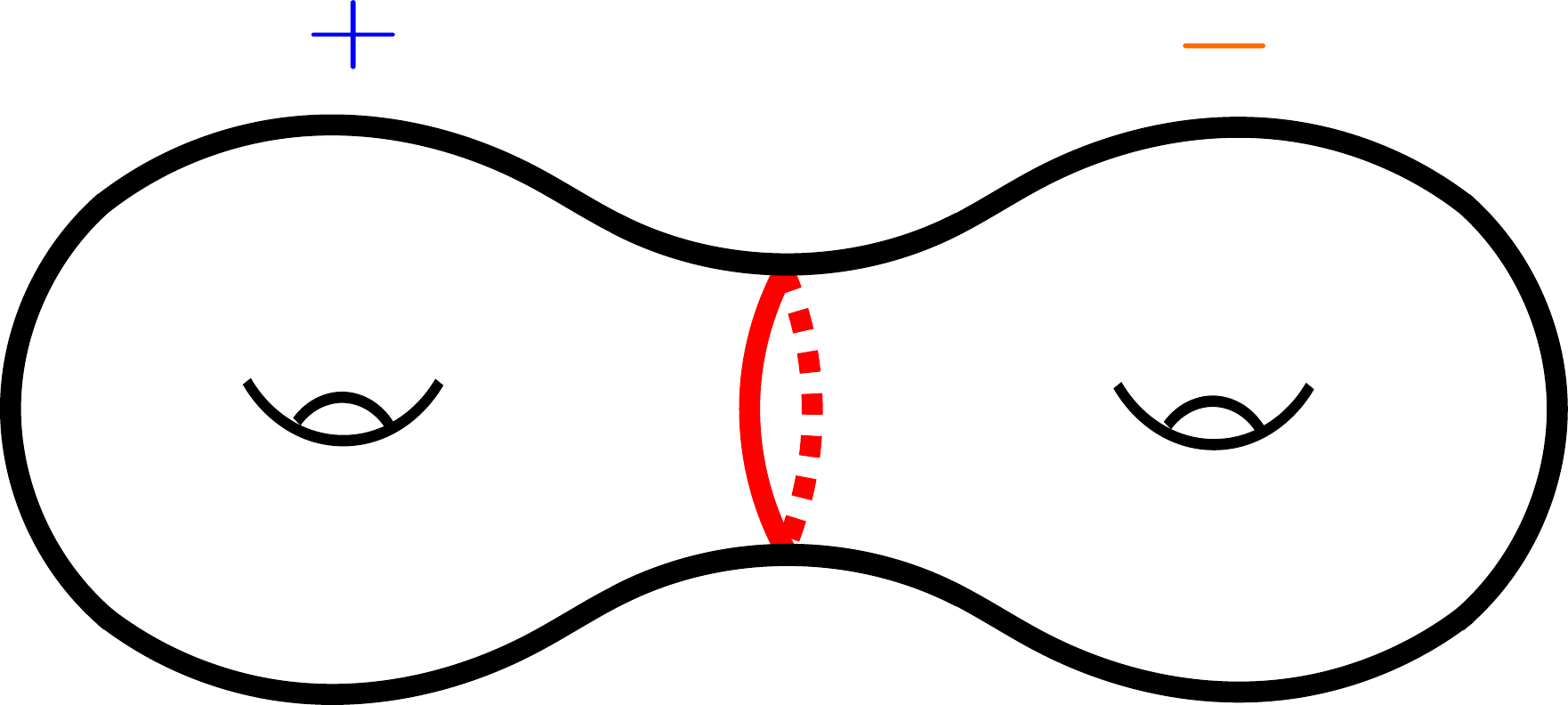}
            \caption{A single vertebrae}
            \label{fig:vertebrae}
        \end{figure}

        The \emph{spine} of the surface $S$ is comprised of multiple pairs of surface path components, and we call each of these pairs \textit{vertebrae}, see Figure \ref{fig:vertebrae}. Each vertebrae is constructed by gluing two 1-genus tori together, and the spine is constructed by gluing each vertebrae together end to end, see Figure \ref{fig:spine}. Note, not all of these gluings are trivial surgeries, and we describe them in more detail shortly. Each torus represents a surface path component of our eventual splitting, and the total number of surface path components is $2(\min\set{\#|S_+|,\#|S_-|}-1)$. Each vertebrae consists of a ``positive" component glued to a ``negative" component, and then the spine is glued in an alternating positive-negative pattern. Therefore, without loss of generality, the spine is a surface with:
        \begin{multicols}{2}
            \begin{itemize}[itemsep=0.1in]
                \item $\#|S_-|-1$ positive components,
                \item $\#|S_-|-1$ negative components,
                \item genus equal to $2(\#|S_-|-1)$, and
                \item will account for $2(\#|S_-|-1)-1$ fold curves.
            \end{itemize}\vspace{0.1in}
        \end{multicols}\

        \begin{figure}[h]
            \centering
            \includegraphics[width=0.75\linewidth]{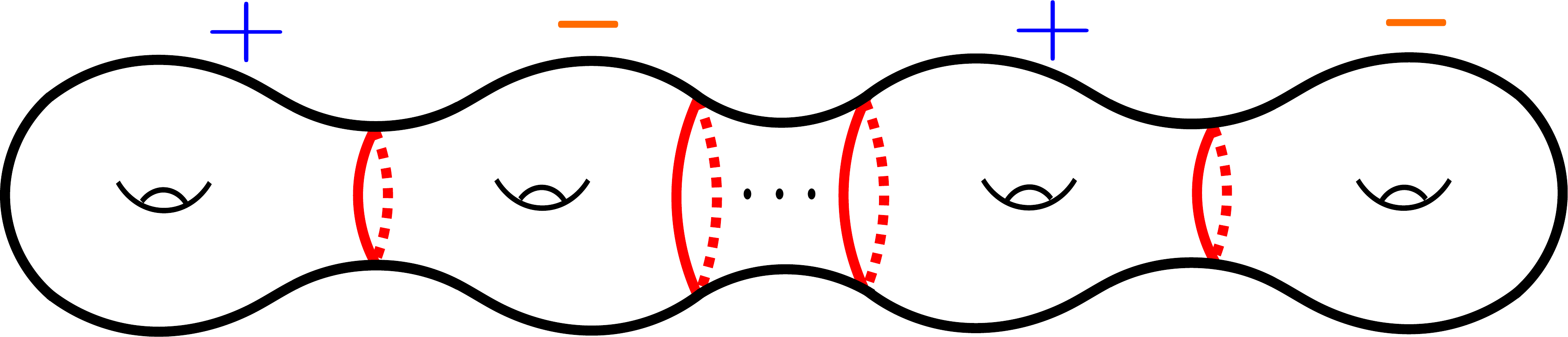}
            \caption{Spine created from multiple vertebrae}
            \label{fig:spine}
        \end{figure}

        The \emph{skull} of the surface $S$ is comprised of a pair of surface path components, one positive and one negative, see Figure \ref{fig:skull}. Each component in this pair is a surface of genus $\frac{1}{2}(g-|\Sigma|+1)$, and they are glued together by first selecting $|\Sigma|+2-(\#|S_+|+\#|S_-|)$ disjoint discs on each component. These discs then serve as attaching regions for 1-handle surgery between the two components. We again note that not all of these are trivial surgeries, and we soon describe them in more detail. The skull is what we call the result of this surgery, and it is a surface with:
            \begin{itemize}[itemsep=0.1in]
                \item $1$ positive component,
                \item $1$ negative component,
                \item genus equal to $g+2-(\#|S_+|+\#|S_-|)$, and
                \item will account for $|\Sigma|+2-(\#|S_+|+\#|S_-|)$ fold curves.
            \end{itemize}\vspace{0.1in}

        Finally, the \emph{teeth} of the surface $S$ is a disjoint collection of $\abs{\#|S_+|-\#|S_-|}$ genus $1$ tori, see Figure \ref{fig:teeth}. Recall, we assumed, without loss of generality, that $\#|S_+|\geq\#|S_-|$, so every component of the teeth will be a positive component.

        \begin{figure*}[h]
            \centering
            \begin{subfigure}{0.5\textwidth}
                \centering
                \includegraphics[width=0.5\linewidth]{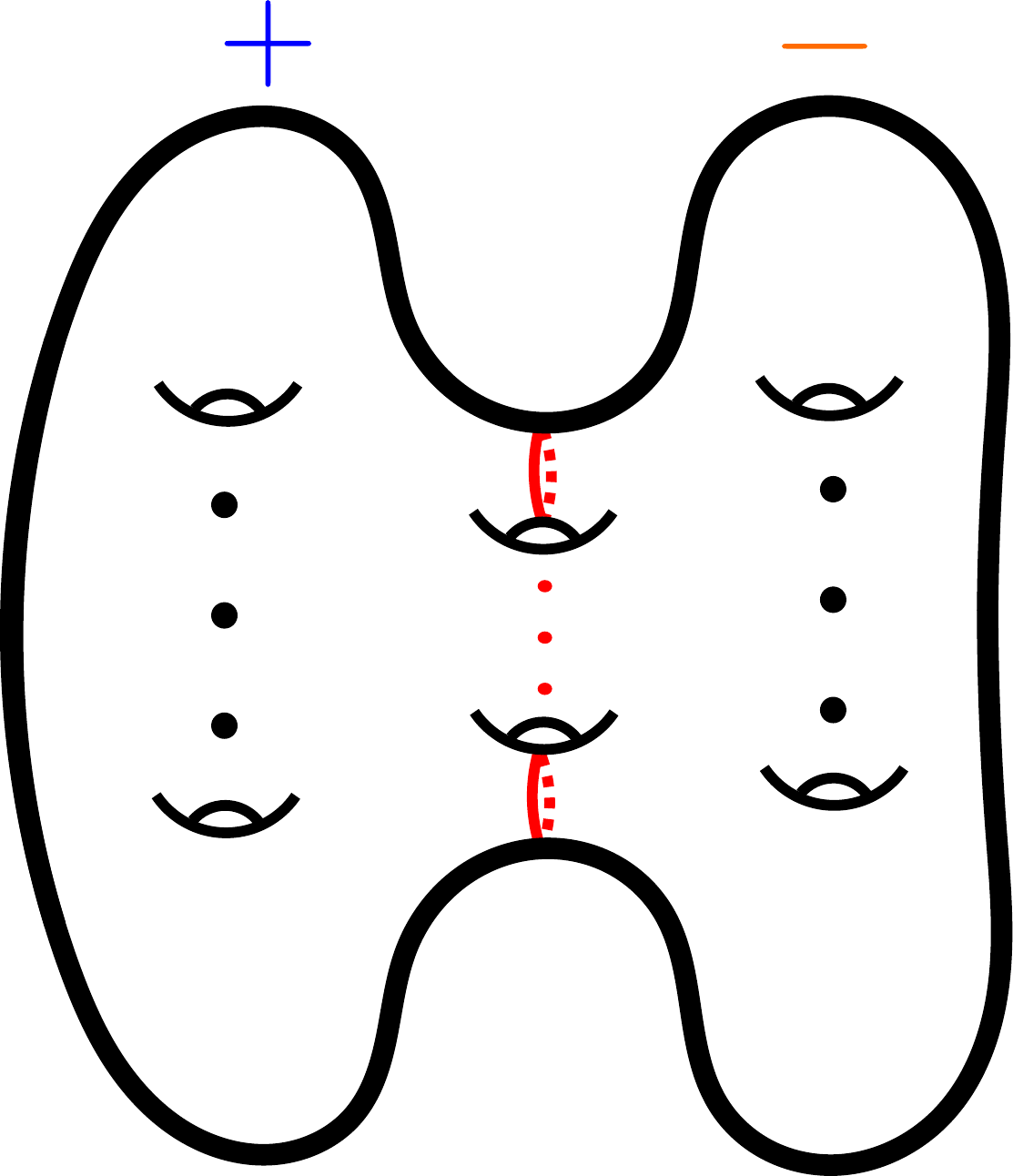}
                \caption{The Skull}
                \label{fig:skull}
            \end{subfigure}%
            ~ 
            \begin{subfigure}{0.5\textwidth}
                \centering
                \includegraphics[width=0.75\linewidth]{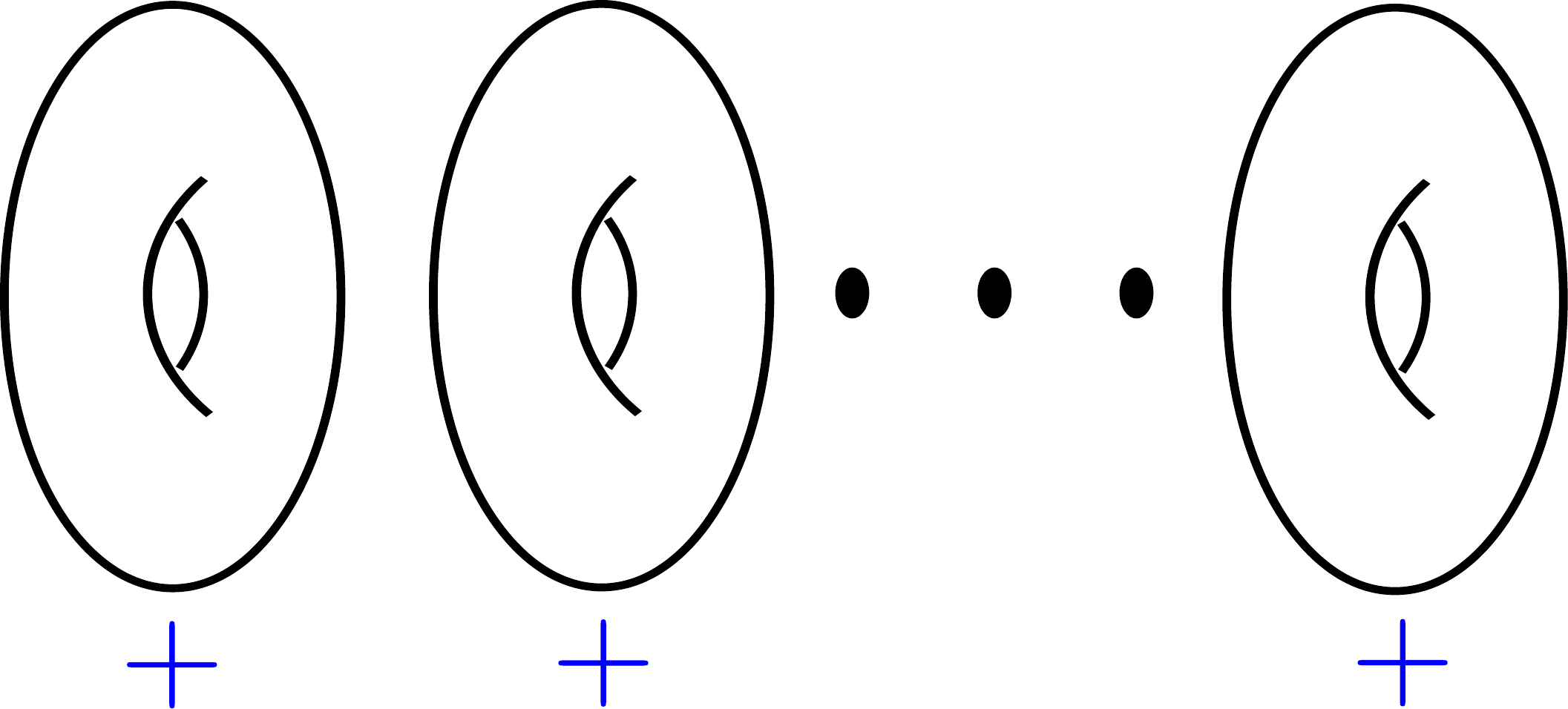}
                \caption{The Teeth}
                \label{fig:teeth}
            \end{subfigure}
            \caption{The Skull and Teeth}
        \end{figure*}
        
        With these sub-surfaces constructed, we now glue them to construct our desired surface $S$. In Figure \ref{fig:main preglue} we see the components before they are glued together, and in Figure \ref{fig:mainsurface} we see the fully constructed surface. First, we glue the teeth to the, without loss of generality, negative component of the skull. Second, without loss of generality, we glue the negative component of the skull to a positive component of the spine that is glued to only one negative component of the spine. These final gluings will account for the remaining $\abs{\#|S_+|-\#|S_-|}+1$ folds. We reiterate, we will discuss the details of these surgeries shortly. Thus, our surface $S$ is constructed and has, without loss of generality:
            \begin{itemize}[itemsep=0.1in]
                \item positive components equal to:
                    \[(\#|S_-|-1)+1+\abs{\#|S_+|-\#|S_-|}=\#|S_+|\]
                \item negative components equal to:
                    \[(\#|S_-|-1)+1=\#|S_-|\]
                \item genus equal to:
                    \[2(\#|S_-|-1)+(g+2-(\#|S_+|+\#|S_-|))+\abs{\#|S_+|-\#|S_-|}=g\]
                \item has a number of folds equal to:
                    \[(2(\#|S_-|-1)-1)+(|\Sigma|+2-(\#|S_+|+\#|S_-|))+(\abs{\#|S_+|-\#|S_-|}+1)=|\Sigma|\]
            \end{itemize}
            
        \begin{figure}[h]
            \centering
            \includegraphics[width=0.75\linewidth]{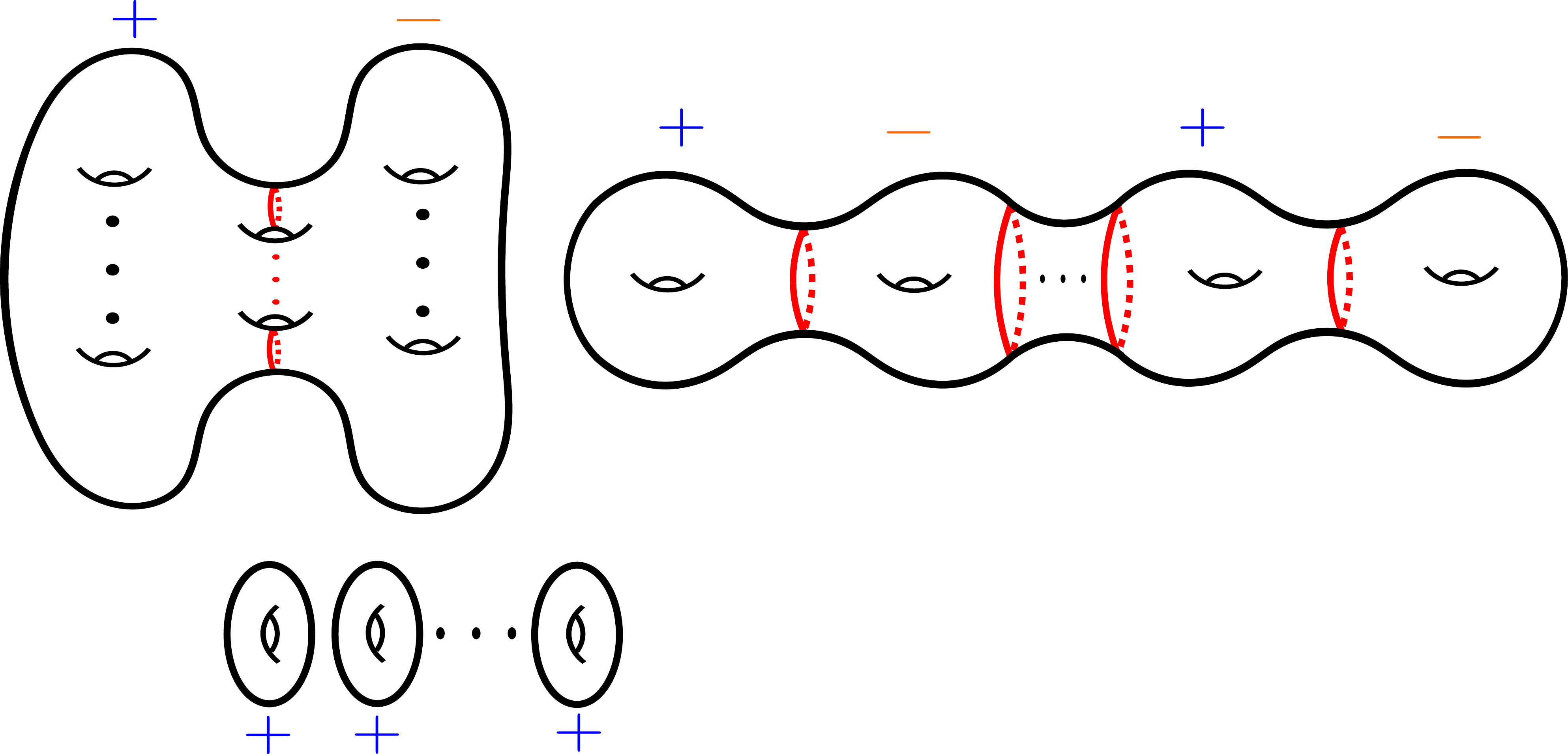}
            \caption{The skull, spine, and teeth separate.}
            \label{fig:main preglue}
        \end{figure}

        \begin{figure}[h]
            \centering
            \includegraphics[width=0.75\linewidth]{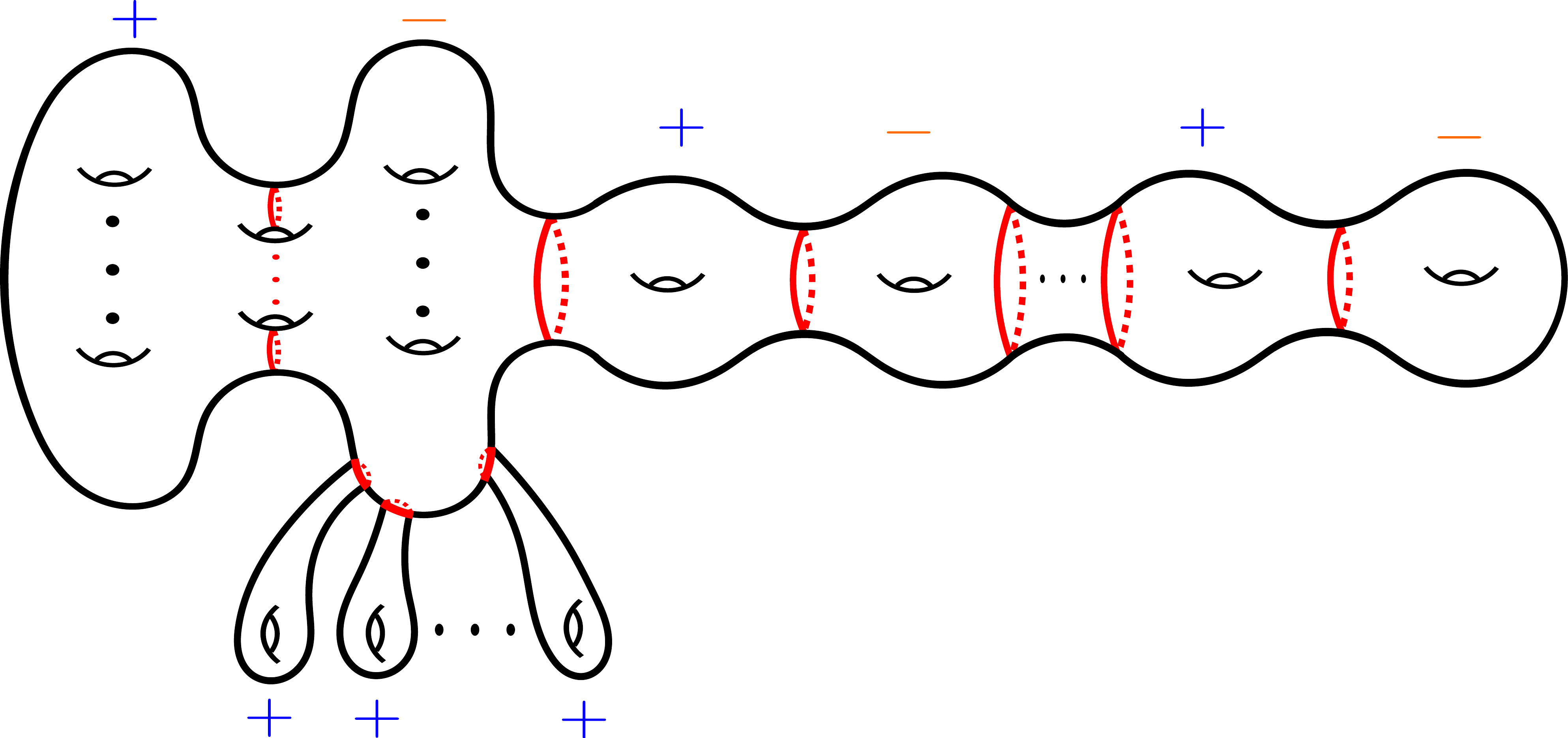}
            \caption{The constructed surface $S$.}
            \label{fig:mainsurface}
        \end{figure}

        With the surface $S$ constructed, we now can define the conditions of a desired simple fold map $f:S\to\R^2$. The desired simple fold map $f$ is one that, without loss of generality: has folds as the cocores of every gluing from our construction above, preserves the orientation on all positive components of $S_+$, and reverses orientation on all negative of components $S_-$. We emphasize that such a map exists by Èliašberg \cite{El70}, since the Euler characteristic of $S_+$ and Euler characteristic of $S_-$ agree. To show this equivalence, we begin by examining the Euler characteristic of $S_+$. Recall, without loss of generality, that $S_+$ consists of $\#|S_+|$ disjoint surfaces with boundary where one component is in the skull, $\abs{\#|S_+|-\#|S_-|}$ components are the teeth, and $\#|S_-|-1$ components are in the spine. These components of $S_+$ have the following properties:
            \begin{itemize}[itemsep=0.1in]
                \item Each positive component in the spine is a surface of genus $1$ with $2$ boundary components.
                \item Each positive component of the teeth is a surface of genus $1$ with $1$ boundary components.
                \item The single positive component of the skull is a surface of genus $\frac{1}{2}(g-|\Sigma|+1)$ with $|\Sigma|+2-(\#|S_+|+\#|S_-|)$ boundary components. 
            \end{itemize}\vspace{0.1in}
        Thus, the disconnected surface $S_+$ has genus $g_+$ and $k_+$ boundary components where, without loss of generality:
            \[g_+=\frac{1}{2}(g-|\Sigma|+1)+\abs{\#|S_+|-\#|S_-|}+\#|S_-|-1=\#|S_+|+\frac{1}{2}(g-|\Sigma|-1)\]
            \[k_+=|\Sigma|+2-(\#|S_+|+\#|S_-|)+\abs{\#|S_+|-\#|S_-|}+2(\#|S_-|-1)=|\Sigma|\]
        Meaning that the Euler characteristic of $S_+$ is:
            \[\chi(S_+)=2(\#|S_+|)-2\left(\#|S_+|+\frac{1}{2}(g-|\Sigma|-1)\right)-|\Sigma|=1-g\]

        Similarly, without loss of generality, recall that $S_-$ consists of $\#|S_-|$ disjoint surfaces with boundary where one component is in the skull and $\#|S_-|-1$ components are in the spine. These components of $S_-$ have the following properties:
            \begin{itemize}[itemsep=0.1in]
                \item One negative component of the spine is a surface of genus $1$ with $1$ boundary component
                \item Each of the other $\#|S_-|-2$ negative components of the spine is surface of genus $1$ with $2$ boundary components.
                \item The single negative component of the skull is a surface of genus $\frac{1}{2}(g-|\Sigma|+1)$ with $|\Sigma|+2-(\#|S_+|+\#|S_-|)+\abs{\#|S_+|-\#|S_-|}+1$ boundary components. 
            \end{itemize}\vspace{0.1in}
        Thus, the disconnected surface $S_-$ has genus $g_-$ and $k_-$ boundary components where, without loss of generality:
            \[g_-=\frac{1}{2}(g-|\Sigma|+1)+\#|S_-|-1=\#|S_-|+\frac{1}{2}(g-|\Sigma|-1)\]
            \[k_-=|\Sigma|+2-(\#|S_+|+\#|S_-|)+\abs{\#|S_+|-\#|S_-|}+1+2(\#|S_-|-2)+1=|\Sigma|\]
        Meaning that the Euler characteristic of $S_-$ is:
            \[\chi(S_-)=2(\#|S_-|)-2\left(\#|S_-|+\frac{1}{2}(g-|\Sigma|-1)\right)-|\Sigma|=1-g\]
        Therefore, such a simple fold map exists where the folds occur at the desired boundaries.
        
        We now show that the self-intersections of the folds are as desired by describing the specifics of the gluings. The surface path components of $S$, mapped by our desired simple fold map $f$, are viewed as immersed surfaces with boundary where these boundaries are the folds. Recall that by the sharpness argument of Theorem \ref{Thm:Self_Int_of_Boundary}, self-intersections of immersed boundaries of surfaces with multiple boundary components can be concentrated on a single immersed boundary, leaving all other boundary components embedded. Additionally, recall that by Theorem \ref{Thm:Self_Int_of_Boundary}, the minimum number of self-intersections of the the immersed boundary of a surface of genus $g$ is $2g+2$. We now define where the self-intersections of the boundaries will occur. Using our construction, with the skull, spine, and teeth, we will concentrate the self-intersections within the creation of the vertebrae, the creation of the skull, and the connection of the teeth to the skull. All other folds will have no self-intersections. 
        
        We begin by examining the spine and vertebrae. By Theorem \ref{Thm:Self_Int_of_Boundary}, the boundaries of a genus-$1$ component must have at least $4$ self-intersections. From our construction, each vertebrae is two, genus-$1$ tori glued together, and then the spine is constructed by gluing vertebrae together. By Proposition \ref{Prop:SpecificBoundaries}, each surface component, except for one negative component, of the spine is constructed as surface of genus one with two boundary components: $B_{3,1}^{+}$ and $C^{+}_{0}$. The remaining negative component has a single boundary component $B_{3,1}^{+}$. In summary, the single fold component \textit{within} each vertebrae has exactly $4$ self-intersections, and the fold components joining the vertebrae together each have $0$ self-intersections. Therefore, since the spine has $2(\#|S_-|-1)$ surface components, the folds within the spine will have a total of $4(\#|S_-|-1)$ self-intersections.

        We now examine the positive component of the skull and the teeth. Recall that the skull was constructed from two surfaces, one positive and one negative, and the positive component is a surface with genus $\frac{1}{2}(g-|\Sigma|+1)$ and $|\Sigma|+2-(\#|S_+|+\#|S_-|)$ boundary components. By Proposition \ref{Prop:SpecificBoundaries}, we construct the boundary of the positive component of the skull to be:
            \[B^{+}_{g-|\Sigma|+1,1}\sqcup\bigsqcup_{i=1}^{|\Sigma|+2-(\#|S_+|+\#|S_-|)-1}C^{+}_{0}\] 
        Moreover, each tooth is a surface of genus one with a single boundary component constructed to be $B^{+}_{3,1}$.

        With the spine, teeth, and positive component of the skull examined, we can now see how all of these components connect together using the negative component of the skull. Recall that the negative component of the skull is a surface of genus $\frac{1}{2}(g-|\Sigma|+1)$ and must have $\left(|\Sigma|+2-(\#|S_+|+\#|S_-|)\right)+\abs{\#|S_+|-\#|S_-|}+1$ boundary components. Note that $\left(|\Sigma|+2-(\#|S_+|+\#|S_-|)\right)$ boundaries are from creating the skull, $\abs{\#|S_+|-\#|S_-|}$ boundaries are from attaching teeth, and the final one is from attaching the skull to the spine. Thus, by Proposition \ref{Prop:SpecificBoundaries} we construct the boundary of the negative component of the skull to be:
            \[B^{+}_{g-|\Sigma|+1,1}\sqcup\left(\bigsqcup_{i=1}^{|\Sigma|+2-(\#|S_+|+\#|S_-|)-1}C^{+}_{0}\right)\sqcup\left(\bigsqcup_{j=1}^{\abs{\#|S_+|-\#|S_-|}} C^{+}_{0}\right)\sqcup C^+_0\] 

        Additionally, recall, from Example \ref{Positive_Isotopy_Contrapositive}, that $C^{+}_{0}$ is positively isotopic to $A^+_{2}$, and $A^+_{2}$ is positively isotopic to $B^+_{3,1}$. Therefore, the boundary of the negative component of the skull is positively isotopic to:
            \[B^{+}_{g-|\Sigma|+1,1}\sqcup\left(\bigsqcup_{i=1}^{|\Sigma|+2-(\#|S_+|+\#|S_-|)-1}C^{+}_{0}\right)\sqcup\left(\bigsqcup_{j=1}^{\abs{\#|S_+|-\#|S_-|}} B^{+}_{3}\right)\sqcup C^+_0\] 
        Therefore, all folds within the skull and teeth have a total of $g-|\Sigma|+3+4\abs{\#|S_+|-\#|S_-|}$ self-intersections. Thus, without loss of generality, the folds of the surface $S$ under $f$ have self-intersection equal to:
            \[\Delta_\Sigma=4(\#|S_-|-1)+g-|\Sigma|+3+4\abs{\#|S_+|-\#|S_-|}=4(\#|S_+|)+g-|\Sigma|-1\]
        Which completes the proof, since, without loss of generality:
            \[2(\rho(G)+1+n)-(\chi(S)/2+|\Sigma|)=4(\#|S_+|)+g-|\Sigma|-1\]
    \end{proof}

\begin{example}\label{MainTheoremExample}
    Given specific values for $\#|S_+|$, $\#|S_-|$, $|\Sigma|$, and $g$, let's construct a surface along with a fold map that admits a number of self-intersection points $\Delta_{\Sigma}$ identical to our sharp lower bound. For this example let us consider 
    \begin{align*}
       \#|S_+| &= 5 \\
       \#|S_-| &= 3 \\
       |\Sigma| &= 9 \\
       g &= 14
    \end{align*}

    A quick calculation yields that, given these values, we aspire to construct a simple stable map that has $\Delta_{\Sigma} = 24$. We realize this example by following the construction of our sharp lower bound shown in Theorem \ref{thm:main theorem sharpness}. We first construct the spine of the desired surface. The spine of the surface will have $\#|S_-| - 1 = 2$ positive components and will have $\#|S_-| - 1 = 2$ negative components. The spine will contribute $2(\#|S_-|-1) = 4$ genus and $2(\#|S_-|-1) - 1 = 3$ fold components. We next turn to construction the skull of the desired surface. The skull will consist of one positive and one negative component. The genus of the skull is then
        \[g+2 - (\#|S_+| + \#|S_-|) = 16 - 8 = 8\] 
    and the number of fold components accounted for is \[|\Sigma|+2-(\#|S_+|+\#|S_-|) = 11 - 8 = 3\] The last anatomy to construct is the teeth, which are $|\#|S_+| - |S_-|| = 2$ tori of genus $1$.

    Next, by the surgery described in the proof of Theorem \ref{thm:main theorem sharpness}, we glue the spine to the skull via the trivial positive concordance from $C_0^+$ to itself, and each tooth is glued to the skull via positive concordance from $C_0^+$ to $B_{3,1}^+$. Recalling that each surgery adds one component to $|\Sigma|$ we obtain a surface, along with a fold map, that has the desired $\#|S_+|= 5$, $\#|S_-|=3$, $|\Sigma|=9$, and $g=14$. This surface is shown in Figure \ref{fig:ExampleSurface}.

    \begin{figure}[h]
        \centering
        \includegraphics[width=0.8\linewidth]{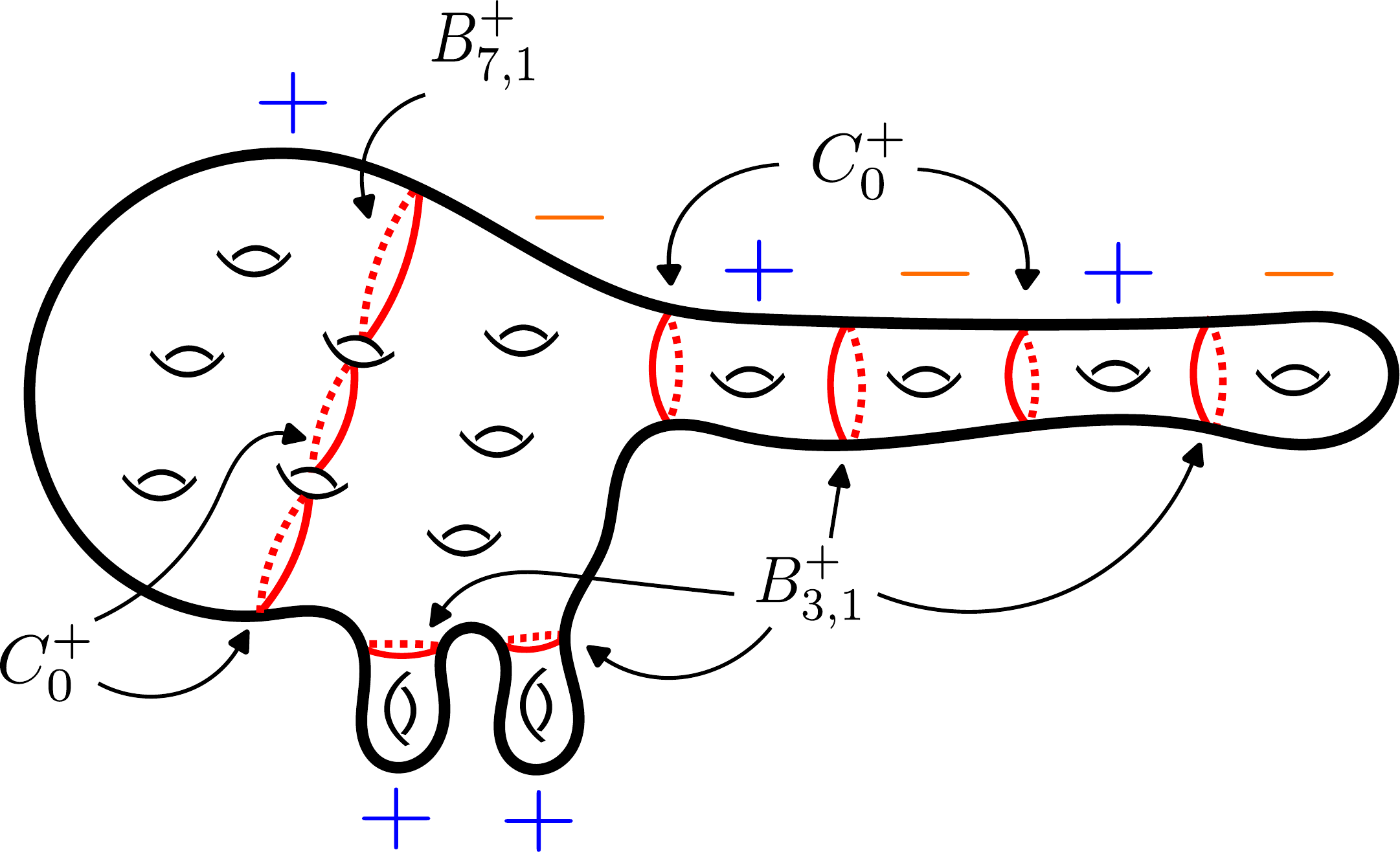}
        \caption{Surface from Example \ref{MainTheoremExample}}
        \label{fig:ExampleSurface}
    \end{figure}
    \begin{figure}[h]
    \centering
    \includegraphics[width=0.5\linewidth]{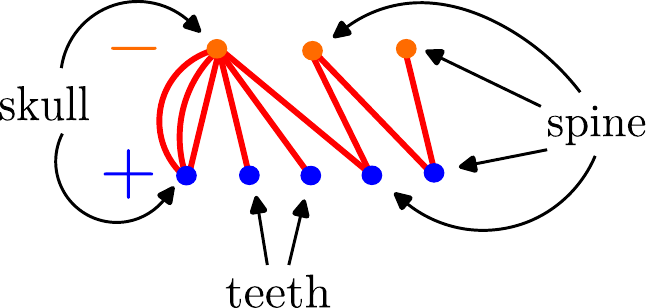}
    \caption{Biparitite graph corresponding to Example \ref{MainTheoremExample}}
    \label{fig:MainExampleGraph}
    \end{figure}

    In Figure \ref{fig:ExampleSurface}, we see the skull consisting of a positive and a negative component whose common boundary is three components of $|\Sigma|$. Note that since these $3$ components are a common boundary for two surface components, the surgery results in adding two genera to the final constructed surface.  Now, since each component of the skull has genus $3$, we note that under a simple fold map, there must be at least $8$ self-intersection points in these 3 components of $|\Sigma|$ within the skull. By the construction used in the proof of Theorem \ref{Thm:Self_Int_of_Boundary}, we have that one of these curves maps to a $B_{7,1}^+$ and the other two map to a $C_0^+$. Next, since there are two teeth (which are both positive components) there will also be $2$ components of $|\Sigma|$ after they are surged to the skull. With each tooth having genus $1$, by Guth \cite{Gu09}, each singular set component bounding a tooth will map to a $B_{3,1}^+$. 
    
    As for the spine, since it is glued to the skull trivially, this surgery adds one component of $|\Sigma|$ that is mapped to $C_0^+$. The spine itself consists of two vertebrae, accounting for two positive surface components and two negative surface components. As shown in the previous theorem, each vertebrae's positive component is glued to its negative component such that the resulting singular set component maps to $B_{3,1}^+$. Then, the trivial gluing of the two vertebrae together accounts for the final singular set component, which is mapped to a $C_0^+$. In totality, the singular set maps to 
        \[
        \underset{\text{the skull}}{\underline{B_{7,1}^+ \sqcup C_0^+ \sqcup C_0^+}} \sqcup \underset{\text{the teeth}}{\underline{B_{3,1}^+ \sqcup B_{3,1}^+}} \sqcup \underset{\text{the spine}}{\underline{C_0^+ \sqcup B_{3,1}^+ \sqcup C_0^+ \sqcup B_{3,1}^+}} 
        \]
    If we count the number of self-intersections of these curves, we get
        \[
        (7+1)+0+0+(3+1)+(3+1)+0+(3+1)+0+(3+1) = 24
        \]
Finally, we corroborate that this agrees with the minimum of our previous theorem

\begin{align*}
    \Delta_\Sigma=4(\#|S_-|-1)+g-|\Sigma|+3+4\abs{\#|S_+|-\#|S_-|} &= 4(\#|S_+|)+g-|\Sigma|-1 \\
    &= 4(5) + 14 - 9 - 1 \\
    &= 24
\end{align*}
and 
\begin{align*}
    2(\rho(G)+1+n)-(\chi(S)/2+|\Sigma|) &= 2(7+1+2) - (-13 + 9) \\
                                        &= 20 - (-4) = 24
\end{align*}
where $G$ is the associated bipartite graph shown in Figure \ref{fig:MainExampleGraph}.

\end{example}

\subsection{Relating Sharp Lower Bound to the Homology of the Surface}

Let $|H_*(S)|$ be the sum of Betti numbers of the closed orientable surface $S$. In \cite[Section 2.1]{Gromov_2009}, Gromov proved the following inequality
\[
\Delta_{\Sigma} \geq \frac{1}{2}|H_*(S)| - |\Sigma|
\]
We wish to compare our result to Gromov's.
Note, for closed, oriented surfaces of finite genus $g$, we have $|H_*(S)| = 2g+2$
\begin{align*}
\chi(S) = 2 - 2g &\Rightarrow 2g + 2 + \chi(S) = 4 \\
&\Rightarrow |H_*(S)| + \chi(S) = 4 
\end{align*}

\noindent Substituting, we see 
    \begin{align*}
    \Delta_{\Sigma} &\geq 2(\rho(G)+n+1) - (\frac{1}{2}\chi(S) + |\Sigma|) \\
        &=2\rho(G)+2n + 2 -\frac{1}{2}(4-|H_*(S)|)-|\Sigma|\\
        &=2(\rho(G)+ n) +\frac{1}{2}|H_*(S)| -|\Sigma|\\
    \end{align*}
This yields the following refinement of Gromov's homological lower bound.

\begin{corollary}\label{coro:gromov extension}
    Let $S$ be a closed orientable surface and $f:S \to \R^2$ a simple fold map. Then, 
        \[
            \Delta_{\Sigma}  \geq 2(\rho(G)+ n) +\frac{1}{2}|H_*(S)| -|\Sigma|
        \]
    where $|\Sigma|$ is the number of path components of fold-singularities, $\rho(G)$ is the number of edges in a spanning tree of the graph $G$ corresponding to $f$, $\Delta_{\Sigma}$ is the number of self-intersections of fold singularities, and $|H_*(S)|$ denotes the sum of the Betti numbers of $S$.
\end{corollary}

\bibliography{references}
\bibliographystyle{plain}

\end{document}